\documentclass[12pt]{amsart}
\usepackage{epsfig,mathrsfs,amssymb,enumerate,comment,esint,graphicx}
\usepackage[dvipsnames]{xcolor}
\usepackage{mathrsfs}
\usepackage{hyperref}
\hypersetup{
    colorlinks=true, 
    linktoc=all,     
    linkcolor=blue,  
}

\theoremstyle{definition}
\newtheorem{theorem}{Theorem}[section]
\newtheorem{definition}[theorem]{Definition}

\newtheorem{lemma}[theorem]{Lemma}
\newtheorem{remark}[theorem]{Remark}
\newtheorem{corollary}[theorem]{Corollary}

\numberwithin{equation}{section}

\title[Modified Mean Curvature Flow in Fuchsian Manifolds]{Modified Mean Curvature Flow in Fuchsian Manifolds}

\author{Yuk Shing Lam}
\address[Y. Lam]{Mathematics Department\\University of California, Santa Cruz\\1156 High Street\\
Santa Cruz, CA 95064\\USA}
\email{ylam14@ucsc.edu}

\author{Longzhi Lin}
\address[L. Lin]{Mathematics Department\\University of California, Santa Cruz\\1156 High Street\\
Santa Cruz, CA 95064\\USA}
\email{lzlin@ucsc.edu}


\begin{document}

\begin{abstract}
In this paper, we show that the modified mean curvature flow starting from an arbitrary graph in a Fuchsian manifold exists for all time and converges smoothly to an equidistant surface of constant mean curvature as $t\to \infty$. This result generalizes earlier work to the modified mean curvature flow setting and removes the restrictive global gradient bound initially required for the standard mean curvature flow by Huang, Zhou, and the second author \cite{HLZ2020}.
\end{abstract}

\maketitle

\tableofcontents

\section{Introduction}

The deformation of submanifolds via curvature-driven flows is a central subject in geometric analysis. These flows are instrumental in exploring the topological and geometric structures of ambient manifolds and are frequently employed to deform submanifolds into optimal configurations, such as minimal or constant mean curvature (CMC) surfaces. A well-known analytical obstacle in this pursuit is the formation of finite-time singularities, which commonly arise under geometric flows such as the standard mean curvature flow due to the avoidance principle. 

In complete hyperbolic 3-manifolds, such as quasi-Fuchsian manifolds, the ambient geometry provides a rich structure of foliations. Understanding the long-time behavior of geometric flows in such spaces yields profound insights into the existence and uniqueness of CMC surfaces. In this article, we focus on the evolution of surfaces within Fuchsian manifolds. A Fuchsian manifold $M^3$ is a complete hyperbolic 3-manifold that can be globally expressed as a warped product of a closed hyperbolic surface $\Sigma$ and the real line $\mathbb{R}$. The reference surface $\Sigma$ sits as the unique totally geodesic slice in $M^3$.

The primary goal of this paper is to study the evolution of geodesic graphs under the modified mean curvature flow \eqref{eqn:MMCF-new}, which incorporates a forcing term to drive the evolving surface toward a surface of prescribed constant mean curvature. Geometrically, this flow arises as the natural negative gradient flow of a functional combining the surface area with the volume enclosed between the surface and a fixed reference slice. It was originally introduced by the second author and Xiao \cite{LX12} to construct CMC hypersurfaces in hyperbolic space $\mathbb{H}^{n+1}$ subject to prescribed asymptotic Dirichlet boundary data at infinity (c.f. the asymptotic Plateau problem). In the unforced setting, the long-time existence and convergence of the standard mean curvature flow for graphs were established by Huang, the second author, and Zhang \cite{HLZ2020}, provided the initial surface satisfies suitable global bounds on its gradient or angle function. We refer to Huang, the second author, and Zhang \cite{HLZ2023} for the study of the modified mean curvature flow in the broader context of almost-Fuchsian manifolds. For earlier related results, we refer to \cite{U03, BM2012, HZZ19}.

In this paper, we show that the modified mean curvature flow starting from an arbitrary smooth geodesic graph in a Fuchsian manifold exists for all time and converges smoothly, as time tends to infinity, to an equidistant surface of constant mean curvature. This result generalizes earlier work to the modified mean curvature flow setting and removes the restrictive global gradient bound initially required for the standard mean curvature flow by Huang, Zhou, and the second author \cite{HLZ2020}. More precisely, we prove
\begin{theorem}\label{main-thm}
Let $M^3$ be a Fuchsian manifold and $\Sigma$ the unique closed totally geodesic surface in $M^3$. For any constant $\sigma \in (-2, 2)$ and any initial smooth closed surface $S_0 \subset M^3$ that is a geodesic graph over $\Sigma$, the modified mean curvature flow \eqref{eqn:MMCF-new} starting from $S_0$ exists for all time, remains a geodesic graph over $\Sigma$, and converges smoothly as $t \to \infty$ to the equidistant surface $\Sigma(r_\sigma)$, where $r_\sigma = \text{arctanh}\left(\frac{\sigma}{2}\right)$.
\end{theorem}

The paper is organized as follows. In Section \ref{Prelim}, we collect the necessary geometric properties of Fuchsian manifolds and formally set up the modified mean curvature flow. We then derive the evolution equation for the height function and establish a height estimate. In Section \ref{sec:linearized-operator}, we introduce the crucial linearized operator, derived from the evolution equation of the height function. Section \ref{sec:time-dep-gra-bound} is devoted to establishing a time-dependent gradient estimate and long-time existence of the flow. Finally, in Section \ref{sec:unif-gra-boud-conv}, we prove the global time-independent uniform gradient bound and the asymptotic smooth convergence of the flow.

\section{Preliminaries}\label{Prelim}
In this section, we establish the notational conventions and recall the basic geometric features of Fuchsian manifolds. We also define the modified mean curvature flow and the precise notion of graphical hypersurfaces in this ambient space.

\subsection{The Ambient Fuchsian Geometry}
Let $\Sigma$ be a closed surface of genus $g \ge 2$ equipped with a hyperbolic metric $g_0$. A Fuchsian manifold $M^3$ is constructed as the warped product $\mathbb{R} \times \Sigma$. We equip $M^3$ with the standard warped product metric:
\begin{equation}
\bar{g} = dr^2 + \cosh^2(r)g_0,
\end{equation}
where $r \in \mathbb{R}$ serves as the coordinate along the real line. Under this metric, $M^3$ is a complete 3-manifold with constant sectional curvature $-1$. 

The level sets $\Sigma(r) = \{r\} \times \Sigma$ form a global equidistant foliation of $M^3$. The slice $\Sigma(0) = \Sigma$ is the unique totally geodesic surface in the manifold. For any $r \neq 0$, the slice $\Sigma(r)$ is a totally umbilic surface with principal curvatures equal to $\tanh(r)$.

We denote the Levi-Civita connection on $M^3$ by $\overline{\nabla}$. The unit normal vector field to the foliation $\Sigma(r)$ is given by $\mathbf{n}: = \frac{\partial}{\partial r}$. A highly useful geometric feature of the Fuchsian manifold is the existence of the conformal vector field $V = \cosh(r) \mathbf{n}$. It can be easily checked that for any tangent vector field $X$ on $M^3$, the covariant derivative of the normal field $\mathbf{n}$ satisfies (see \cite[Lemma 2.1]{HLZ2020}):
\begin{equation}\label{eqn:conformal-vf}
\overline{\nabla}_X \mathbf{n} = \tanh(r)(X - \bar{g}(X, \mathbf{n})\mathbf{n}).
\end{equation}

\subsection{Modified Mean Curvature Flow and Graphical Surfaces}
Let $S$ be a closed surface. We consider a smooth family of immersions $F: S \times [0, T) \to M^3$. The modified mean curvature flow is defined by the initial value problem:
\begin{equation}\label{eqn:MMCF-new}
\begin{cases}
\frac{\partial}{\partial t}F(x,t) = -(H(x,t) - \sigma)\nu(x,t), \\
F(\cdot, 0) = F_0,
\end{cases}
\end{equation}
where $H(x,t)$ is the mean curvature of the evolving surface $S(t) = F(S, t)$, $\nu(x,t)$ is the chosen unit normal vector field on $S(t)$, and $\sigma \in (-2,2)$ is a given real constant. Here, we adopt the convention that the mean curvature is the trace of the second fundamental form. 

To analyze the geometry of $S(t)$ relative to the foliation of the ambient space, we define the height function $u(x,t)$ as the signed distance from the point $F(x,t)$ to the totally geodesic reference surface $\Sigma$:
\begin{equation}
u(x,t) = r(F(x,t)).
\end{equation}
The relative orientation of the surface $S(t)$ is measured by the angle function $\Theta(x,t)$, defined as the inner product between the upward-pointing vertical vector field $\mathbf{n}$ and the surface normal $\nu$:
\begin{equation}\label{eqn:angle-function}
\Theta(x,t) = \bar{g}(\mathbf{n}, \nu)(F(x,t)).
\end{equation}

With these quantities established, we formalize the definition of a graphical surface in this setting.

\begin{definition} \label{def:graph}
A closed, smooth surface $S_0 \subset M^3$ is defined as a geodesic graph over the totally geodesic surface $\Sigma$ if there exists a constant $c_0 > 0$ such that the angle function satisfies $\Theta_0 = \bar{g}(\mathbf{n}, \nu_0) \ge c_0$ everywhere on $S_0$. 
\end{definition}

If $\Theta(x,t) \in (0, 1]$ everywhere, the evolving surface $S(t)$ can be represented globally by the height function $r = u(x,t)$ over the reference surface $\Sigma$. A surface is a constant height slice $\Sigma(r)$ if and only if $\Theta \equiv 1$.
\vskip 2mm

Let $D$, $\nabla$, and $\bar{\nabla}$ denote the Levi-Civita connections on the fixed reference surface $(\Sigma, g_0)$, the evolving surface $(S(t), g(t))$, and the ambient manifold $(M^3, \bar{g})$, respectively. At any \textit{fixed} time $t$, we parameterize the evolving surface $$\Sigma_t:= S(t)$$ as a graph over $\Sigma$ via the immersion $F(x) = (u(x), x)$ for $x \in \Sigma$. Choosing a local coordinate system $(x^1, x^2)$ on $\Sigma$ with the coordinate frame $\left\{\frac{\partial}{\partial x^1}, \frac{\partial}{\partial x^2}\right\}$, the covariant derivatives of the height function $u: \Sigma \to \mathbb{R}$ are denoted by:
\[
u_i = D_i u , \qquad u_{ij} = D_i D_j u, \qquad \text{ and }\qquad |Du|^2:= |Du|_{g_0}^2=(g_0)^{ij}u_iu_j.
\]
Then
\[
F_i=\frac{\partial F}{\partial x^i}=\frac{\partial u}{\partial x^i}\frac{\partial }{\partial r}+\frac{\partial }{\partial x^i}=u_i\frac{\partial }{\partial r}+\frac{\partial }{\partial x^i}.
\]
In the following, we denote $\left\langle \cdot,\cdot\right\rangle:= \left\langle \cdot,\cdot\right\rangle_{\bar{g}}$. Then the induced metric is
\begin{align*}
g_{ij}&=\left\langle F_i,F_j\right\rangle \\
&=\left\langle u_i\frac{\partial }{\partial r}+\frac{\partial }{\partial x^i},u_j\frac{\partial }{\partial r}+\frac{\partial }{\partial x^j}\right\rangle \\
&=u_i u_j\left\langle \frac{\partial }{\partial r}, \frac{\partial }{\partial r}\right\rangle + u_i\left\langle\frac{\partial }{\partial r} ,\frac{\partial }{\partial x^j}\right\rangle + u_j\left\langle\frac{\partial }{\partial x^i} ,\frac{\partial }{\partial r}\right\rangle + \left\langle\frac{\partial }{\partial x^i} ,\frac{\partial }{\partial x^j}\right\rangle.
\end{align*}
Since the ambient metric is $\bar{g}=dr^2+\cosh^2(r) g_0$, we have
\[
\left\langle \frac{\partial}{\partial r},\frac{\partial}{\partial r}\right\rangle=1,\quad\left\langle \frac{\partial}{\partial r},\frac{\partial}{\partial x^i}\right\rangle=0,\quad\left\langle \frac{\partial}{\partial x^i},\frac{\partial}{\partial x^j}\right\rangle=\cosh^2(r)(g_0)_{ij}.
\]
Thus, evaluating at $r=u$,
\begin{equation}
g_{ij}=u_iu_j+\cosh^2(u)(g_0)_{ij}.
\end{equation}
On the other hand, the inverse metric $g^{ij}$ can be deduced from
\[
\delta_{\ell}^j=g_{\ell k}g^{kj}.
\]
Contracting this identity with $(g_0)^{i \ell}$ gives
\begin{align*}
(g_0)^{i \ell} \delta_{\ell}^j&=(g_0)^{i \ell} g_{\ell k} g^{kj} \\
(g_0)^{ij}&=(g_0)^{i \ell}\left(u_{\ell} u_k + \cosh^2(u)(g_0)_{\ell k}\right) g^{kj} \\
&=u^i u_k g^{kj}+\cosh^2(u)g^{ij} \\
&=u^i \alpha^j+\cosh^2(u)g^{ij},
\end{align*}
where $\alpha^j :=u_k g^{kj}$, and hence
\begin{equation}\label{eqn:g_ij}
g^{ij} = \frac{1}{\cosh^2(u)}\left((g_0)^{ij}-u^i \alpha^j\right).
\end{equation}
To determine $\alpha^j$, we contract the identity $\delta_{\ell}^j=g_{\ell k}g^{kj}$ with $u^{\ell}$ this time:
\begin{align*}
u^\ell \delta_\ell^j&=u^{\ell} g_{\ell k} g^{kj} \\
u^j&=u^{\ell} \left(u_{\ell} u_k + \cosh^2(u)(g_0)_{\ell k}\right) g^{kj} \\
&=u^{\ell} u_{\ell} u_k g^{kj}+\cosh^2(u)u^{\ell} (g_0)_{\ell k} g^{kj}.
\end{align*}
Note that $u^{\ell} u_{\ell}=|Du|^2$ and $u^{\ell} (g_0)_{\ell k}=u_k$, so we obtain
\begin{align*}
u^j&=|Du|^2 \alpha^j+\cosh^2(u)\alpha^j \\
&=\alpha^j\left(|Du|^2+\cosh^2(u)\right).
\end{align*}
Therefore
\begin{equation}\label{eqn:alpha-j}
u_k g^{kj} = \alpha^j=\frac{u^j}{|Du|^2+\cosh^2(u)}.
\end{equation}
Substituting into \eqref{eqn:g_ij}, we obtain
\begin{align}\label{eqn:g^ij}
    g^{ij}=\frac{1}{\cosh^2(u)} \left((g_0)^{ij} - \frac{u^i u^j}{|Du|^2+\cosh^2(u)}\right).
\end{align}
Let 
\begin{equation}\label{eqn:gra-function}
w:=\sqrt{1+\frac{|Du|^2}{\cosh^2 (u)}}
\end{equation}
be the gradient function. Since $w^2 \cosh^2(u) = \cosh^2(u) + |Du|^2$, we have 
\begin{equation}\label{eqn:g-upper-ij}
g^{ij}=\frac{1}{\cosh^2(u)}\left((g_0)^{ij}-\frac{u^iu^j}{w^2\cosh^2(u)} \right).
\end{equation}

Now let $X$ be a tangent vector field on $\Sigma$, naturally extended to $M^3$ such that it commutes with the radial vector field, meaning the Lie bracket $\left[\frac{\partial}{\partial r}, X\right] = 0$. Using \eqref{eqn:conformal-vf} with $\mathbf{n} = \frac{\partial}{\partial r}$, we have
\begin{equation}\label{eqn:connection-1}
\bar\nabla_{\frac{\partial}{\partial r}}\frac{\partial}{\partial r} = \tanh(r)\left( \frac{\partial}{\partial r} - \bar{g}\left(\frac{\partial}{\partial r}, \frac{\partial}{\partial r}\right)\frac{\partial}{\partial r} \right) = 0.
\end{equation}
Furthermore, since $X$ is tangent to $\Sigma$, it is orthogonal to the radial direction ($\bar{g}(X, \mathbf{n}) = 0$). Therefore,
\[
\bar\nabla_{\frac{\partial}{\partial r}}X = \bar\nabla_X\frac{\partial}{\partial r} = \tanh(r)\left( X - \bar{g}\left(X, \frac{\partial}{\partial r}\right)\frac{\partial}{\partial r} \right) = \tanh(r)X.
\]
In particular,
\begin{equation}\label{eqn:connection-2}
\bar\nabla_{\frac{\partial}{\partial r}}\frac{\partial}{\partial x^j} = \bar\nabla_{\frac{\partial}{\partial x^j}}\frac{\partial}{\partial r} = \tanh(r)\frac{\partial}{\partial x^j}.
\end{equation}

Next, we compute the covariant derivative $\bar{\nabla}_{\frac{\partial}{\partial x^i}}\frac{\partial}{\partial x^j}$. The metric components in the local coordinate frame $\{r, x^1, x^2\}$ are:
\[
\bar{g}_{rr} = 1, \qquad \bar{g}_{ri} = 0, \qquad \bar{g}_{ij} = \cosh^2(r)(g_0)_{ij}.
\]
The inverse metric components are:
\[
\bar{g}^{rr} = 1, \qquad \bar{g}^{ri} = 0, \qquad \bar{g}^{ij} = \frac{1}{\cosh^2(r)}(g_0)^{ij}.
\]
Expanding $\bar{\nabla}_{\frac{\partial}{\partial x^i}}\frac{\partial}{\partial x^j}$ into its radial and spatial components, we have
\[
\bar{\nabla}_{\frac{\partial}{\partial x^i}}\frac{\partial}{\partial x^j} = \bar{\Gamma}_{ij}^r \frac{\partial}{\partial r} + \bar{\Gamma}_{ij}^k \frac{\partial}{\partial x^k}.
\]
For $\bar{\Gamma}_{ij}^r$, we compute:
\begin{align*}
\bar{\Gamma}_{ij}^r = \frac{1}{2} \bar{g}^{rr} \left( \frac{\partial \bar{g}_{jr}}{\partial x^i} + \frac{\partial \bar{g}_{ir}}{\partial x^j} - \frac{\partial \bar{g}_{ij}}{\partial r} \right) = - \cosh(r)\sinh(r)(g_0)_{ij}.
\end{align*}
For $\bar{\Gamma}_{ij}^k$, where $k \in \{1, 2\}$, we have:
\begin{align*}
\bar{\Gamma}_{ij}^k &= \frac{1}{2} \bar{g}^{k\ell} \left( \frac{\partial \bar{g}_{j\ell}}{\partial x^i} + \frac{\partial \bar{g}_{i\ell}}{\partial x^j} - \frac{\partial \bar{g}_{ij}}{\partial x^\ell} \right) \\
&= \frac{1}{2} \left[ \frac{1}{\cosh^2(r)} (g_0)^{k\ell} \right] \left( \cosh^2(r)\frac{\partial (g_0)_{j\ell}}{\partial x^i} + \cosh^2(r)\frac{\partial (g_0)_{i\ell}}{\partial x^j} - \cosh^2(r)\frac{\partial (g_0)_{ij}}{\partial x^\ell} \right)\\
&= \frac{1}{2} (g_0)^{k\ell} \left( \frac{\partial (g_0)_{j\ell}}{\partial x^i} + \frac{\partial (g_0)_{i\ell}}{\partial x^j} - \frac{\partial (g_0)_{ij}}{\partial x^\ell} \right) \\
&= (\Gamma_0)_{ij}^k,
\end{align*}
where $(\Gamma_0)_{ij}^k$ are the Christoffel symbols of the metric $g_0$ on the refence surface $\Sigma$. Using that $D_{\frac{\partial}{\partial x^i}}\frac{\partial}{\partial x^j} = (\Gamma_0)_{ij}^k \frac{\partial}{\partial x^k}$, we have
\begin{align}\label{eqn:connection-3}
\bar{\nabla}_{\frac{\partial}{\partial x^i}}\frac{\partial}{\partial x^j} &= \bar{\Gamma}_{ij}^k \frac{\partial}{\partial x^k} + \bar{\Gamma}_{ij}^r \frac{\partial}{\partial r} \notag\\
&= D_{\frac{\partial}{\partial x^i}}\frac{\partial}{\partial x^j} - \cosh(r)\sinh(r)(g_0)_{ij}\frac{\partial}{\partial r}.
\end{align}

We now compute the second fundamental form
\begin{equation}\label{eqn:2ndFF}
h_{ij}=-\langle \bar\nabla_{F_i}F_j,\nu\rangle.
\end{equation}
Recall that
\[
F_i=u_i\frac{\partial}{\partial r}+\frac{\partial}{\partial x^i}.
\]
Then
\begin{align*}
\bar\nabla_{F_i}F_j&=\bar\nabla_{u_i\frac{\partial}{\partial r}+\frac{\partial}{\partial x^i}} \left(u_j\frac{\partial}{\partial r}+\frac{\partial}{\partial x^j}\right) \\
&=u_i u_j \bar\nabla_{\frac{\partial}{\partial r}}\frac{\partial}{\partial r} +u_i \bar\nabla_{\frac{\partial}{\partial r}}\frac{\partial}{\partial x^j} +u_j \bar\nabla_{\frac{\partial}{\partial x^i}}\frac{\partial}{\partial r} +\bar\nabla_{\frac{\partial}{\partial x^i}}\frac{\partial}{\partial x^j} +(\partial_i u_j)\frac{\partial}{\partial r}.
\end{align*}
Using \eqref{eqn:connection-1}, \eqref{eqn:connection-2},\eqref{eqn:connection-3} and evaluating at $r=u$, this simplifies to
\begin{align}\label{eqn:2ndFF-1}
\bar\nabla_{F_i}F_j= & u_i\tanh(u)\frac{\partial}{\partial x^j} +u_j\tanh(u)\frac{\partial}{\partial x^i} +D_{\frac{\partial}{\partial x^i}}\frac{\partial}{\partial x^j} \notag\\
& -\cosh(u)\sinh(u)(g_0)_{ij}\frac{\partial}{\partial r} +(\partial_i u_j)\frac{\partial}{\partial r}.
\end{align}
Now let $\tilde{\nu} = \frac{\partial}{\partial r} + B^k \frac{\partial}{\partial x^k}$ be an unnormalized normal vector, then we have
\begin{align*}
0 &= \left\langle \tilde{\nu}, F_i \right\rangle \\
&= \left\langle \frac{\partial}{\partial r} + B^k \frac{\partial}{\partial x^k}, u_i \frac{\partial}{\partial r} + \frac{\partial}{\partial x^i} \right\rangle \\
&= u_i \left\langle \frac{\partial}{\partial r}, \frac{\partial}{\partial r} \right\rangle + B^k \left\langle \frac{\partial}{\partial x^k}, \frac{\partial}{\partial x^i} \right\rangle \\
&= u_i + B^k \cosh^2(u) (g_0)_{ki} \\
&= u_i + \cosh^2(u) B_i\,,
\end{align*}
which yields $B_i = -\frac{u_i}{\cosh^2(u)}$, and raising the index gives $B^k = -\frac{u^k}{\cosh^2(u)}$. Thus,
\[
\tilde{\nu} = \frac{\partial}{\partial r} - \frac{u^k}{\cosh^2(u)} \frac{\partial}{\partial x^k}
\]
with
\begin{align*}
|\tilde{\nu}|^2 &= \left\langle \tilde{\nu}, \tilde{\nu} \right\rangle \\
&= \left\langle \frac{\partial}{\partial r}, \frac{\partial}{\partial r} \right\rangle + \frac{u^k u^\ell}{\cosh^4(u)} \left\langle \frac{\partial}{\partial x^k}, \frac{\partial}{\partial x^\ell} \right\rangle \\
&= 1 + \frac{u^k u_k}{\cosh^2(u)} = 1 + \frac{|Du|^2}{\cosh^2(u)} = w^2.
\end{align*}
Therefore, the unit outward normal to $S(t)$ is
\begin{equation}\label{eqn:normal-vf-to-S_t}
\nu = \frac{1}{w} \left( \frac{\partial}{\partial r} - \frac{u^k}{\cosh^2(u)} \frac{\partial}{\partial x^k} \right).
\end{equation}
Taking the inner product of the basis vectors with $\nu$ yields the relation between the angle function $\Theta(x,t)$ (see \eqref{eqn:angle-function}) with the gradient function $w$:
\begin{align}
\Theta = \left\langle \frac{\partial}{\partial r},\nu\right\rangle&=\frac{1}{w}\,, \label{eqn:angle-gradient}\\
\left\langle \frac{\partial}{\partial x^i},\nu\right\rangle &= \frac{1}{w}\left\langle \frac{\partial}{\partial x^i}, \frac{\partial}{\partial r} - \frac{u^k} {\cosh^2(u)}\frac{\partial}{\partial x^k} \right\rangle \label{eqn:angle-gradient2} \\
&= -\frac{1}{w\cosh^2(u)} u^k \cosh^2(u)(g_0)_{ik} = -\frac{u_i}{w}.\notag
\end{align}

Now we evaluate \eqref{eqn:2ndFF}: $h_{ij} = -\langle \bar\nabla_{F_i}F_j,\nu\rangle$ term by term using \eqref{eqn:2ndFF-1} and \eqref{eqn:normal-vf-to-S_t}. For the symmetric mixed terms, by \eqref{eqn:angle-gradient2}, we have
\begin{align*}
&-\left\langle u_i\tanh(u)\frac{\partial}{\partial x^j} +u_j\tanh(u)\frac{\partial}{\partial x^i}, \nu \right\rangle \\
=& -u_i\tanh(u)\left(-\frac{u_j}{w}\right) - u_j\tanh(u)\left(-\frac{u_i}{w}\right) 
= \frac{2\tanh(u)u_i u_j}{w}.
\end{align*}
Next note that we have $D_{\frac{\partial}{\partial x^i}}\frac{\partial}{\partial x^j} = (\Gamma_0)_{ij}^k \frac{\partial}{\partial x^k}$. Therefore, along with \eqref{eqn:angle-gradient} and \eqref{eqn:angle-gradient2}, we obtain:
\begin{align*}
-\left\langle D_{\frac{\partial}{\partial x^i}}\frac{\partial}{\partial x^j} + (\partial_i u_j)\frac{\partial}{\partial r}, \nu \right\rangle &= -(\Gamma_0)_{ij}^k\left(-\frac{u_k}{w}\right) - (\partial_i u_j)\left(\frac{1}{w}\right) \\
&= -\frac{1}{w}\left(\partial_i u_j - (\Gamma_0)_{ij}^k u_k\right) = -\frac{u_{ij}}{w}.
\end{align*}
Finally, for the remaining normal term we have
\[
-\left\langle -\cosh(u)\sinh(u)(g_0)_{ij}\frac{\partial}{\partial r}, \nu \right\rangle = \frac{\cosh(u)\sinh(u)(g_0)_{ij}}{w}.
\]
Combining these yields:
\begin{equation}\label{eqn:2ndFF-2}
h_{ij} = \frac{1}{w} \left( \cosh(u)\sinh(u)(g_0)_{ij} + 2\tanh(u)u_i u_j - u_{ij} \right).
\end{equation}
Raising the indices twice, we get
\begin{align*}
h^{ij}&=g^{ik}g^{j\ell}h_{k\ell} \\
&=\frac{1}{w}\left(\cosh(u)\sinh(u)g^{ik}g^{j\ell}(g_0)_{k\ell}+2\tanh(u)(g^{ik}u_k)(g^{j\ell}u_\ell)-g^{ik}g^{j\ell}u_{k\ell}\right).
\end{align*}
For the first term:
\begin{align*}
g^{ik}g^{j\ell}(g_0)_{k\ell} &= g^{ik}\frac{1}{\cosh^2(u)}\left(\delta^j_k - \frac{u^j u_k}{w^2\cosh^2(u)}\right) \\
&= \frac{1}{\cosh^2(u)}g^{ij} - \frac{u^j}{w^2\cosh^4(u)}(g^{ik}u_k) \\
&= \frac{1}{\cosh^2(u)}g^{ij} - \frac{u^i u^j}{w^4\cosh^6(u)}.
\end{align*}
Note that we previously established that $g^{ik}u_k = \frac{u^i}{w^2\cosh^2(u)}$ (see \eqref{eqn:alpha-j}).
Substituting these back, we have
\begin{equation}
\begin{aligned}[b]
h^{ij}&=\frac{1}{w}\Bigg[\cosh(u)\sinh(u)\left(\frac{1}{\cosh^2(u)}g^{ij}-\frac{u^i u^j}{w^4\cosh^6(u)}\right)\\
&\qquad + 2\tanh(u)\frac{u^i}{w^2\cosh^2(u)}\frac{u^j}{w^2\cosh^2(u)} - g^{ik}g^{j\ell}u_{k\ell}\Bigg] \\
&=\frac{1}{w}\left[\tanh(u)g^{ij}-\frac{\tanh(u)}{w^4\cosh^4(u)}u^i u^j+\frac{2\tanh(u)}{w^4\cosh^4(u)}u^i u^j-g^{ik}g^{j\ell}u_{k\ell}\right] \\
&=\frac{1}{w}\left[\tanh(u)g^{ij}+\frac{\tanh(u)}{w^4\cosh^4(u)}u^i u^j-g^{ik}g^{j\ell}u_{k\ell}\right].
\end{aligned}
\end{equation}
Now the mean curvature is
\begin{align*}
H&=g^{ij}h_{ij}  \\
&=\frac{1}{w}\left(\cosh(u)\sinh(u)g^{ij}(g_0)_{ij}+2\tanh(u)g^{ij}u_i u_j-g^{ij}u_{ij}\right).
\end{align*}
Since $\Sigma$ is a 2-dimensional surface, we have $(g_0)^{ij}(g_0)_{ij} = 2$. Therefore,
\begin{align*}
g^{ij}(g_0)_{ij}&=\frac{1}{\cosh^2(u)}\left((g_0)^{ij}(g_0)_{ij}-\frac{u^i u^j (g_0)_{ij}}{w^2\cosh^2(u)}\right) \\
&=\frac{1}{\cosh^2(u)}\left(2-\frac{|Du|^2}{w^2\cosh^2(u)}\right).
\end{align*}
Using $|Du|^2 = \cosh^2(u)(w^2-1)$, this simplifies to:
\begin{equation}\label{eqn:mismatchedgs}
g^{ij}(g_0)_{ij}=\frac{1}{\cosh^2(u)}\left(2-\frac{w^2-1}{w^2}\right)=\frac{1}{\cosh^2(u)}\left(1+\frac{1}{w^2}\right).   
\end{equation}
We also have
\begin{align}\label{eqn:guiuj}
\notag g^{ij}u_i u_j&=\frac{1}{\cosh^2(u)}\left((g_0)^{ij}u_i u_j-\frac{u^i u^j u_i u_j}{w^2\cosh^2(u)}\right) \\
&=\frac{1}{\cosh^2(u)}\left(|Du|^2-\frac{|Du|^4}{w^2\cosh^2(u)}\right) \notag \\
&=(w^2-1)-\frac{(w^2-1)^2}{w^2} =\frac{w^2-1}{w^2}.
\end{align}
Therefore, 
\begin{align*}
H&=\frac{1}{w}\left[\cosh(u)\sinh(u)\frac{1}{\cosh^2(u)}\left(1+\frac{1}{w^2}\right)+2\tanh(u)\frac{w^2-1}{w^2}-g^{ij}u_{ij}\right] \\
&=\frac{1}{w}\left[\tanh(u)\left(1+\frac{1}{w^2}\right)+2\tanh(u)\left(1-\frac{1}{w^2}\right)-g^{ij}u_{ij}\right] \\
&=\frac{1}{w}\left[\tanh(u)\left(3-\frac{1}{w^2}\right)-g^{ij}u_{ij}\right].
\end{align*}
Now the evolution equation of $u$ is 
\begin{align}\label{eqn:u_t}
\frac{\partial u}{\partial t}&=-(H-\sigma)\Theta \notag\\
&=-\frac{H-\sigma}{w} \notag\\
&=\frac{1}{w^2}g^{ij}u_{ij}-\frac{\tanh(u)}{w^2}\left(3-\frac{1}{w^2}\right)+\frac{\sigma}{w}.
\end{align}
This is a degenerate second order quasi-linear parabolic PDE. Therefore a uniform upper bound on $w$ implies the angle function $\Theta \in (0, 1]$ everywhere (see \eqref{eqn:angle-function} and \eqref{eqn:angle-gradient}) and the evolving surface $S(t)$ exists for all time and can be represented globally by the height function $u(\cdot,t)$ over the reference surface $\Sigma$, see also Remark \ref{rmk:uniform-higher-order-bound}.

The following lemma serves as the Squeezing Lemma (analogous to \cite[Theorem 3.1]{HLZ2020}) for the modified mean curvature flow, see also \cite[Theorem 4.3]{HLZ2023}.
\begin{lemma}[Height Estimate]\label{lem:height_bound}
Let $u \in C^{2,1}(\Sigma \times [0,T))$ be a smooth solution to the modified mean curvature flow equation \eqref{eqn:u_t} with constant $\sigma \in (-2,2)$ and initial height satisfying $r_{min} \le u(\cdot, 0) \le r_{max}$. Then, as long as the modified mean curvature flow exists, the height function $u(\cdot, t)$ of the evolving surface satisfies:
\begin{equation}
r_-(t) \le u(\cdot, t) \le r_+(t),
\end{equation}
where $r_-(t)$ and $r_+(t)$ are the solutions to the ODE
\begin{equation} \label{ODE_barrier}
\frac{dr}{dt} = \sigma - 2\tanh(r(t)),
\end{equation}
with initial conditions $r_-(0) = r_{min}$ and $r_+(0) = r_{max}$, respectively. In particular, we have
\begin{equation}
|r_\pm(t) - r_\sigma| \le M e^{- \frac{4-\sigma^2}{2} t},
\end{equation}
for all $t \ge 0$, where $r_\sigma := \tanh^{-1}\left(\frac{\sigma}{2}\right)$ and $M > 0$ is a constant depending only on $\sigma, r_{min},$ and $r_{max}$. 
\end{lemma}
\begin{proof}
The equidistant surfaces $\Sigma(r)$ of the Fuchsian manifold $M^3$ from the central reference surface $\Sigma$ form a global foliation of $M^3$. Each fiber $\Sigma(r)$ is totally umbilic with constant principal curvatures equal to $\tanh(r)$.

If we evolve an equidistant surface $\Sigma(r_0)$ under the modified mean curvature flow with $\sigma\in (-2,2)$, it remains an equidistant surface $\Sigma(r(t))$ for all $t \ge 0$, and its height $r(t)$ evolves according to the ODE:
\begin{equation}\label{eqn:ODE-barrier}
\frac{dr}{dt} = -(H - \sigma) = \sigma - 2\tanh(r(t)).
\end{equation}

By assumption, the initial surface $S_0$ lies in between two equidistant surfaces $\Sigma(r_{min})$ and $\Sigma(r_{max})$. By the avoidance principle for the modified mean curvature flow, the evolving surface $S(t)$ must remain in between the evolving surfaces $\Sigma(r_-(t))$ and $\Sigma(r_+(t))$ for as long as the flow exists. This yields the uniform height bound $r_-(t) \le u(\cdot, t) \le r_+(t)$. 

Solving \eqref{eqn:ODE-barrier} explicitly, we have
\begin{equation}
|\sinh(r_\pm(t) - r_\sigma)| e^{\frac{\sigma}{2} r_\pm(t)} = C_\pm e^{- \frac{4-\sigma^2}{2} t},
\end{equation}
where $C_\pm = |\sinh(r_\pm(0) - r_\sigma)| e^{\frac{\sigma}{2} r_\pm(0)}$ are constants determined by the initial conditions. Because the flows of equidistant surfaces are monotonic, $r_\pm(t)$ is bounded between its initial value and $r_\sigma$. Consequently, the factor $e^{-\frac{\sigma}{2} r_\pm(t)}$ is uniformly bounded. This implies that as $t \to \infty$, the term $|\sinh(r_\pm(t) - r_\sigma)|$ decays exponentially to zero. Since there exists a constant $c > 0$ such that $|\sinh(x)| \ge c|x|$ for $x$ in a bounded interval, we have
\begin{equation}\notag
|r_\pm(t) - r_\sigma| \le M e^{- \frac{4-\sigma^2}{2} t},
\end{equation}
for all $t \ge 0$, where $M > 0$ is a constant depending only on the initial data.
\end{proof}

\section{Linearized Operator}\label{sec:linearized-operator}

In this section, movtivated by the work of the second author and Xiao \cite{LX12}, we linearize the evolution equation for the height function $u \in C^{2,1}(\Sigma \times [0,T])$. The resulting linearized operator will be applied to the gradient function $w$ in the subsequent section to establish a time-dependent gradient bound for $u$. From \eqref{eqn:u_t}, $u$ satisfies a nonlinear parabolic equation of the form
\begin{equation}\label{eqn:flow-eqn}
\mathcal{P}[u] = \frac{\partial u}{\partial t} - \mathcal{F}(D^2 u, Du, u) = 0,
\end{equation}
where the operator $\mathcal{F}$ is given by
\begin{equation}\label{eqn:operator-F}
\mathcal{F}(D^2 u, Du, u) = \frac{1}{w^2}g^{ij}u_{ij} - \frac{\tanh(u)}{w^2}\left(3-\frac{1}{w^2}\right) + \frac{\sigma}{w}.
\end{equation}

To define the linearized operator, we introduce a one-parameter family of variations $u_\epsilon = u + \epsilon \phi$, where $\phi \in C^\infty(\Sigma)$, and compute the Fréchet derivative
\[
\mathcal{L}_u[\phi] = \left.\frac{d}{d\epsilon}\right|_{\epsilon=0} \mathcal{P}[u+\epsilon\phi].
\]
Since the covariant derivative is linear, the variations of $Du$ and $D^2 u$ are given by $D\phi$ and $D^2 \phi$, respectively. Because $\mathcal{F}$ depends on the gradient $Du$ and the function $u$ through the highly nonlinear composite terms $w$ and $g^{ij}$, we define the background-dependent geometric coefficients via partial derivatives:
\[
\alpha^{ij} := \frac{\partial \mathcal{F}}{\partial u_{ij}}, 
\qquad
\beta^k := \frac{\partial \mathcal{F}}{\partial u_k}, 
\qquad
\gamma := \frac{\partial \mathcal{F}}{\partial u},
\]
so that the \textit{linearized operator} takes the precise form
\[
\mathcal{L}_u[\phi]
=
\frac{\partial \phi}{\partial t}
- \alpha^{ij} D_i D_j \phi
- \beta^k D_k \phi
- \gamma \phi.
\]
Here, it is clear that the principal symbol is $$\alpha^{ij}= \frac{1}{w^2}g^{ij}.$$ This operator is parabolic because $\alpha^{ij}(x,t) = \frac{1}{w^2(x,t)}g^{ij}(x,t)$ is strictly positive definite on $\Sigma \times [0,T)$. The drift vector $\beta^k = \frac{\partial \mathcal{F}}{\partial u_k}$ can be computed explicitly using \eqref{eqn:operator-F}. Using the substitution 
\begin{equation}\label{eqn:rho}
    \rho \equiv w^2\cosh^2(u) = \cosh^2(u) + |Du|^2,
\end{equation}
we can write $\alpha^{ij}$ as:
\begin{equation}\label{eqn:alpha-ij}
\alpha^{ij} = \frac{1}{\rho} (g_0)^{ij} - \frac{1}{\rho^2} u^i u^j.
\end{equation}
Since $\frac{\partial \rho}{\partial u_k} = 2u^k$ and $u^i = (g_0)^{ik} u_k$, we compute:
\begin{align*}
\frac{\partial \alpha^{ij}}{\partial u_k} &= - \frac{1}{\rho^2} \frac{\partial \rho}{\partial u_k} (g_0)^{ij} - \frac{1}{\rho^2} \frac{\partial (u^i u^j)}{\partial u_k} + \frac{2}{\rho^3} \frac{\partial \rho}{\partial u_k} u^i u^j \\
&= - \frac{2u^k}{\rho^2} (g_0)^{ij} - \frac{(g_0)^{ik}u^j + u^i(g_0)^{jk}}{\rho^2} + \frac{4u^k}{\rho^3} u^i u^j.
\end{align*}
Contracting this with $u_{ij}$ and substituting $\rho = w^2\cosh^2(u)$ back in yields:
\begin{equation}\label{eqn:alpha_u_k}
\frac{\partial \alpha^{ij}}{\partial u_k} u_{ij} = - \frac{2u^k}{w^4\cosh^4(u)} (g_0)^{ij} u_{ij} - \frac{2(g_0)^{ik} u^j u_{ij}}{w^4\cosh^4(u)} + \frac{4u^k u^i u^j u_{ij}}{w^6\cosh^6(u)}.
\end{equation}
For the lower-order terms in $\mathcal{F}$, we use $\frac{\partial w}{\partial u_k} = \frac{u^k}{w\cosh^2(u)}$. We compute:
\begin{align}\label{eqn:remainder_F_u_k}
&\frac{\partial}{\partial u_k} \left[ - \tanh(u)\left(\frac{3}{w^2} - \frac{1}{w^4}\right) + \frac{\sigma}{w} \right] \\
= & - \tanh(u) \left( -\frac{6}{w^3} + \frac{4}{w^5} \right) \frac{u^k}{w\cosh^2(u)} - \frac{\sigma}{w^2} \frac{u^k}{w\cosh^2(u)} \notag\\
= & \frac{2\tanh(u)u^k}{w^4\cosh^2(u)}\left(3 - \frac{2}{w^2}\right) - \frac{\sigma u^k}{w^3\cosh^2(u)}.\notag
\end{align}
Combining \eqref{eqn:alpha_u_k} and \eqref{eqn:remainder_F_u_k}, we have
\begin{align}\label{eqn:beta-k}
\beta^k = & \frac{\partial }{\partial u_k}\left(\frac{1}{w^2}g^{ij}u_{ij} - \frac{\tanh(u)}{w^2}\left(3-\frac{1}{w^2}\right) + \frac{\sigma}{w}\right) \\
= & - \frac{2u^k}{w^4\cosh^4(u)} (g_0)^{ij} u_{ij} - \frac{2(g_0)^{ik} u^j u_{ij}}{w^4\cosh^4(u)} + \frac{4u^k u^i u^j u_{ij}}{w^6\cosh^6(u)} \notag\\
& + \frac{2\tanh(u)u^k}{w^4\cosh^2(u)}\left(3 - \frac{2}{w^2}\right) - \frac{\sigma u^k}{w^3\cosh^2(u)}.\notag
\end{align}
The zeroth-order term $\gamma = \frac{\partial \mathcal{F}}{\partial u}$ can be computed similarly. Using \eqref{eqn:rho}, \eqref{eqn:alpha-ij} and $\frac{\partial \rho}{\partial u} = 2\cosh(u)\sinh(u) = 2\tanh(u)\cosh^2(u)$, we have
\begin{align*}
\frac{\partial \left(\frac{1}{w^2}g^{ij}\right)}{\partial u} = \frac{\partial \alpha^{ij}}{\partial u} &= - \frac{1}{\rho^2} (2\tanh(u)\cosh^2(u)) (g_0)^{ij} + \frac{2}{\rho^3} (2\tanh(u)\cosh^2(u)) u^i u^j \\
&= - \frac{2\tanh(u)}{w^4\cosh^2(u)} (g_0)^{ij} + \frac{4\tanh(u)}{w^6\cosh^4(u)} u^i u^j.
\end{align*}
Contracting with $u_{ij}$ yields
\[
\frac{\partial \alpha^{ij}}{\partial u} u_{ij} = - \frac{2\tanh(u)}{w^4\cosh^2(u)} (g_0)^{ij} u_{ij} + \frac{4\tanh(u)}{w^6\cosh^4(u)} u^i u^j u_{ij}.
\]
We now differentiate the lower-order terms of $\mathcal{F}$ in \eqref{eqn:operator-F} with respect to $u$. Since $\frac{\partial (w^2)}{\partial u} = \frac{\partial}{\partial u} \left( 1 + |Du|^2 \cosh^{-2}(u) \right) = -2\tanh(u)(w^2-1)$, we have $\frac{\partial w}{\partial u} = - \frac{\tanh(u)(w^2-1)}{w}$. We compute
\begin{align*}
&\frac{\partial}{\partial u} \left[ - \tanh(u)\left(\frac{3}{w^2} - \frac{1}{w^4}\right) \right] \\
= & - \text{sech}^2(u)\left(\frac{3}{w^2} - \frac{1}{w^4}\right) - \tanh(u) \left( -\frac{6}{w^3} + \frac{4}{w^5} \right) \left( - \frac{\tanh(u)(w^2-1)}{w} \right) \\
=& - \text{sech}^2(u)\left(\frac{3}{w^2} - \frac{1}{w^4}\right) - \frac{2\tanh^2(u)(w^2-1)}{w^4} \left(3 - \frac{2}{w^2}\right).
\end{align*}
Differentiating the last term in $\mathcal{F}$ yields $\frac{\partial}{\partial u} (\sigma w^{-1}) = -\sigma w^{-2} \frac{\partial w}{\partial u} = \frac{\sigma \tanh(u)(w^2-1)}{w^3}$. 
Combining all these yields
\begin{align}\label{eqn:gamma}
\gamma =& - \frac{2\tanh(u)}{w^4\cosh^2(u)} (g_0)^{ij} u_{ij} + \frac{4\tanh(u)}{w^6\cosh^4(u)} u^i u^j u_{ij} \\
& - \text{sech}^2(u)\left(\frac{3}{w^2} - \frac{1}{w^4}\right) - \frac{2\tanh^2(u)(w^2-1)}{w^4} \left(3 - \frac{2}{w^2}\right) \notag\\
&+ \frac{\sigma\tanh(u)(w^2-1)}{w^3}.\notag
\end{align}

The derivation of the time-depedent gradient bounds will rely on the differential part of the linearized operator (comprising both the second-order and first-order terms) given by:
\begin{equation}\label{eqn:operator-scr-L}
\mathscr{L} := \frac{\partial}{\partial t} - \alpha^{ij} D_i D_j - \beta^k D_k,
\end{equation}
where the principal symbol $\alpha^{ij} = \frac{1}{w^2}g^{ij}$ guarantees uniform parabolicity as long as the gradient $w$ remains bounded, and the drift vector $\beta^k$ is determined by the background geometry and the Hessian $u_{ij}$ as derived above.

\section{Time-dependent Gradient Bound and Long-time Existence }\label{sec:time-dep-gra-bound}
In this section, we utilize the differential part of the linearized operator $\mathcal{L}_u$ derived in Section \ref{sec:linearized-operator}, namely, $\mathcal{L}$ in \eqref{eqn:operator-scr-L}, to establish a time-dependent gradient bound for the height function along the flow, which yields the long-time existence of the flow. We begin by computing the time derivative for the gradient function $w=\sqrt{1+\frac{|Du|^2}{\cosh^2 (u)}}$ along the flow.
\begin{align*}
\frac{\partial w}{\partial t}&= \frac{1}{2w}\frac{\partial}{\partial t}(w^2)\\
&= \frac{1}{2w}\frac{\partial}{\partial t} \left(1 + \frac{|Du|^2}{\cosh^2(u)}\right) \\
&= \frac{1}{2w}
\left[
\frac{1}{\cosh^2(u)}\frac{\partial}{\partial t}(|Du|^2)
+
|Du|^2 \frac{\partial}{\partial t}\left(\frac{1}{\cosh^2(u)}\right)
\right].
\end{align*}

Since
\[
|Du|^2 = u^k u_k = (g_0)^{ik} u_i u_k,
\]
by symmetry we have
\begin{align*}
\frac{\partial}{\partial t}(|Du|^2)
&= \frac{\partial}{\partial t}\left[(g_0)^{ik} u_i u_k\right] \\
&= 2 (g_0)^{ik} u_i \frac{\partial u_k}{\partial t} \\
&= 2 u^k D_k(u_t).
\end{align*}
Also,
\[
\frac{\partial}{\partial t}\left(\frac{1}{\cosh^2(u)}\right) = -\frac{2\tanh(u)}{\cosh^2(u)}\,u_t.
\]
Substituting these into the evolution of $w$, we obtain
\begin{align*}
\frac{\partial w}{\partial t}
&= \frac{1}{2w}
\left[ \frac{2 u^k}{\cosh^2(u)} D_k(u_t) - |Du|^2 \frac{2\tanh(u)}{\cosh^2(u)} u_t \right] \\
&= \frac{u^k}{w\cosh^2(u)} D_k(u_t) - \frac{1}{w}\frac{|Du|^2}{\cosh^2(u)} \tanh(u)\,u_t.
\end{align*}
Using $ \frac{|Du|^2}{\cosh^2(u)} = w^2 - 1$,
we conclude
\begin{align}\label{eqn:w_t with D_ku_t}
\frac{\partial w}{\partial t}&= \frac{u^k}{w\cosh^2(u)} D_k(u_t) - \frac{w^2 - 1}{w}\,\tanh(u)\,u_t .
\end{align}
Note that by \eqref{eqn:u_t}, $u_t=\frac{1}{w^2}g^{ij}u_{ij}-\frac{\tanh(u)}{w^2}\left(3-\frac{1}{w^2}\right)+\frac{\sigma}{w}$, so we have
\begin{align}\label{eqn:D_ku_t}
D_k(u_t)&= D_k\left(\frac{1}{w^2}g^{ij}u_{ij}\right)-D_k\left(\frac{\tanh(u)}{w^2}\left(3-\frac{1}{w^2}\right)\right)+ \sigma D_k\left(\frac{1}{w} \right).
\end{align}
We now compute each term separately. For the first term, we have
\[
D_k\left(\frac{1}{w^2} g^{ij} u_{ij}\right)=D_k\left(\frac{1}{w^2}\right) g^{ij} u_{ij} + \frac{1}{w^2} D_k(g^{ij})\, u_{ij} + \frac{1}{w^2} g^{ij} D_k(u_{ij}).
\]
Since
\begin{equation}\label{eqn: D-1-w}
D_k\left(\frac{1}{w^2}\right) = -\frac{2}{w^3} w_k \text{\quad and \quad }D_k(u_{ij})=u_{ijk} \ ,
\end{equation}
we obtain
\[
D_k\left( \frac{1}{w^2} g^{ij} u_{ij} \right) = -\frac{2}{w^3} w_k\, g^{ij} u_{ij} + \frac{1}{w^2} D_k(g^{ij})\, u_{ij} + \frac{1}{w^2} g^{ij} u_{ijk}.
\]
Next, noting that $D_k\left(\tanh(u)\right)=\text{sech}^2(u)\, u_k$ and using \eqref{eqn: D-1-w},
\begin{align*}
    D_k\left(\frac{\tanh(u)}{w^2}\left(3-\frac{1}{w^2}\right)\right)
    = & \text{sech}^2(u)\,u_k\,\frac{1}{w^2}\left(3-\frac{1}{w^2}\right) \\
    &
    +\tanh(u)\left(-\frac{2}{w^3}w_k\right)\left(3-\frac{1}{w^2}\right) \\
    &+\tanh(u)\,\frac{1}{w^2}\left(\frac{2}{w^3}w_k\right)\\
    = & \text{sech}^2(u)\,u_k\,\frac{1}{w^2}\left(3-\frac{1}{w^2}\right) - \frac{2\tanh(u)}{w^3}\left(3-\frac{2}{w^2}\right)w_k.
\end{align*}
Lastly we have $D_k\left(\frac{1}{w}\right)=-\frac{1}{w^2}w_k$. Substituting all of these into \eqref{eqn:D_ku_t}, we get
\begin{align}\label{eqn:D_ku_t with gij}
D_k(u_t)&= -\frac{2}{w^3} w_k\, g^{ij} u_{ij} + \frac{1}{w^2} D_k(g^{ij})\, u_{ij} + \frac{1}{w^2} g^{ij} u_{ijk} \\
&\quad \notag
- \text{sech}^2(u)\,\frac{1}{w^2}\left(3-\frac{1}{w^2}\right)u_k \\
&\quad \notag
+ \frac{2\tanh(u)}{w^3}\left(3-\frac{2}{w^2}\right)w_k - \frac{\sigma}{w^2}w_k.
\end{align}
We next compute $D_k(g^{ij})$. By \eqref{eqn:g-upper-ij}, namely, $g^{ij}=\text{sech}^2(u)\,(g_0)^{ij}-\text{sech}^{4}(u)\,\frac{u^i u^j}{w^2}$. Therefore,
\[
D_k(g^{ij})=(g_0)^{ij}D_k\left(\text{sech}^2(u)\right)-\frac{u^i u^j}{w^2}\,D_k\left(\text{sech}^4(u)\right)-\text{sech}^4(u)\,D_k\left(\frac{u^i u^j}{w^2}\right).
\]
Since $D_k\left(\text{sech}^2(u) \right)=-2\text{sech}^2(u)\tanh(u)\,u_k=-\frac{2\tanh (u)}{\cosh^2(u)}\,u_k$, we get $D_k\left(\text{sech}^4(u)  \right)=2\text{sech}^2(u)D_k\left(\text{sech}^2(u) \right)=-\frac{4\tanh (u)}{\cosh^4(u)}\,u_k$. On the other hand, we have
\begin{align*}
    D_k\left(\frac{u^i u^j}{w^2}\right)
    &=\frac{1}{w^2}D_k\left({u^i u^j}\right)+u^i u^j D_k\left(\frac{1}{w^2}\right)=\frac{u^i_k u^j + u^i u^j_k}{w^2}-\frac{2u^i u^j w_k}{w^3}.
\end{align*}
Therefore,
\begin{align*} 
D_k(g^{ij})
&= -2\tanh(u)\cosh^{-2}(u)\,u_k\,(g_0)^{ij}+\frac{4\tanh(u)}{w^2\cosh^{4}(u)}\,u_k\,u^i u^j  \\
&\quad -\frac{1}{w^2\cosh^{4}(u)}\,(u^i_k\,u^j + u^i u^j_k) +\frac{2}{w^3\cosh^{4}(u)}\,w_k\,u^i u^j .
\end{align*}
Substituting this into $\frac{1}{w^2}D_k(g^{ij})u_{ij}$ and contracting with $u_{ij}$ yields $u^i_k u^j u_{ij} + u^i u^j_k u_{ij} = 2u^i_k u^j u_{ij}$. Thus, equation \eqref{eqn:D_ku_t with gij} becomes:
\begin{align*}
D_k(u_t)&= -\frac{2}{w^3} w_k \, g^{ij} u_{ij} -\frac{2 \tanh(u)}{w^2 \cosh^{2}(u)} \, u_k \, (g_0)^{ij} u_{ij} \\
&\quad +\frac{4 \tanh(u)}{w^4 \cosh^{4}(u)} \, u_k \, u^i u^j u_{ij} -\frac{2}{w^4 \cosh^{4}(u)} \, u^i_k \, u^j u_{ij} \\
&\quad +\frac{2}{w^5 \cosh^{4}(u)} \, w_k \, u^i u^j u_{ij} +\frac{1}{w^2} g^{ij} D_k u_{ij} \\
&\quad -\text{sech}^2(u)\,\frac{1}{w^2}\left(3-\frac{1}{w^2}\right) u_k +\frac{2 \tanh(u)}{w^3}\left(3-\frac{2}{w^2}\right) w_k-\frac{\sigma}{w^2} w_k.
\end{align*}
Now we can put the expression on the right side and \eqref{eqn:u_t} into \eqref{eqn:w_t with D_ku_t}.
\begin{align}\label{eqn:w_t-new}
\frac{\partial w}{\partial t} &= \frac{u^k}{w\cosh^2(u)} \Bigg[ -\frac{2}{w^3} w_k \, g^{ij} u_{ij} -\frac{2 \tanh(u)}{w^2 \cosh^{2}(u)} \, u_k \, (g_0)^{ij} u_{ij} \\
&\qquad +\frac{4 \tanh(u)}{w^4 \cosh^{4}(u)} \, u_k \, u^i u^j u_{ij}-\frac{2}{w^4 \cosh^{4}(u)} \, u^i_k \, u^j u_{ij} \notag \\
&\qquad +\frac{2}{w^5 \cosh^{4}(u)} \, w_k \, u^i u^j u_{ij} +\frac{1}{w^2} g^{ij} D_k u_{ij} \notag \\
&\qquad -\text{sech}^2(u)\,\frac{1}{w^2}\left(3-\frac{1}{w^2}\right) u_k +\frac{2 \tanh(u)}{w^3}\left(3-\frac{2}{w^2}\right) w_k -\frac{\sigma}{w^2} w_k \Bigg] \notag \\
&\qquad -\frac{w^2 - 1}{w}\,\tanh(u) \left[ \frac{1}{w^2}g^{ij}u_{ij} -\frac{\tanh(u)}{w^2}\left(3-\frac{1}{w^2}\right) +\frac{\sigma}{w} \right].\notag
\end{align}

\subsection{Time-Dependent Exponential Gradient Bound}
\begin{lemma}[Gradient Evolution Inequality]\label{lem:w_evolution}
Let $u \in C^{2,1}(\Sigma \times [0,T))$ be a smooth solution to the modified mean curvature flow \eqref{eqn:u_t}. There exists a constant $C_0 > 0$, depending only on the background metric $g_0$ on $\Sigma$ and the $C^0$ bound of $u$ in Lemma \ref{lem:height_bound}, such that the gradient function $w$ satisfies the differential inequality:
\begin{equation}
\mathscr{L} w \le C_0 w - \frac{1}{2 w \cosh^2(u)} \alpha^{ij} (g_0)^{k\ell} u_{ik} u_{j\ell} + \frac{1}{w} \alpha^{ij} w_i w_j + 4\tanh(u) \alpha^{ij} u_i w_j,
\end{equation}
where $\mathscr{L}$ is the linear parabolic operator from \eqref{eqn:operator-scr-L}. In particular, we have
\begin{equation}\label{eqn:w_bound_linear}
\mathscr{L} w \le C_0 w+ \frac{1}{w} \alpha^{ij} w_i w_j + 4\tanh(u) \alpha^{ij} u_i w_j.
\end{equation}
\end{lemma}

\begin{proof}
To compute $w_k$, we differentiate $w^2 = 1+\frac{u^\ell u_\ell}{\cosh^2 (u)}$:
\begin{align}
2ww_k &= \frac{2u^\ell u_{\ell k }}{\cosh^2 (u)}-2\tanh(u) \frac{u^\ell u_\ell}{\cosh^2 (u)}u_k \notag\\
w_k &= \frac{u^\ell u_{\ell k }}{w\cosh^2 (u)}-\frac{\tanh(u)}{w} \frac{u^\ell u_\ell}{\cosh^2 (u)}u_k \notag\\
&= \frac{u^\ell u_{\ell k }}{w\cosh^2 (u)}-\frac{\tanh(u)}{w} (w^2-1)u_k. \label{eqn:w_k}
\end{align}
Differentiating once more yields the Hessian $w_{ij} = D_j w_i$:
\begin{align*}
w_{ij} &= D_j\left( \frac{u^\ell u_{\ell i}}{w\cosh^2(u)} \right) - D_j\left( \frac{\tanh(u)}{w}(w^2-1)u_i \right).
\end{align*}
The first term gives:
\begin{align*}
D_j\left( \frac{u^\ell u_{\ell i}}{w\cosh^2(u)} \right) &= \left[-\frac{w_j}{w^2\cosh^2(u)} -\frac{2\tanh(u)}{w\cosh^2(u)}u_j\right]u^\ell u_{\ell i} + \frac{1}{w\cosh^2(u)} (u^\ell_j u_{\ell i} + u^\ell D_j u_{\ell i}).
\end{align*}
Expanding the second term yields:
\begin{align*}
D_j\left( \frac{\tanh(u)}{w}(w^2-1)u_i \right) &= \left(\frac{\text{sech}^2(u)}{w}u_j -\frac{\tanh(u)}{w^2}w_j \right)(w^2-1)u_i \\
&\quad +\frac{\tanh(u)}{w}\,u_i (2ww_j) +\frac{\tanh(u)}{w}(w^2-1)u_{ij}\\
&= \left(\frac{w^2-1}{w}\right)\text{sech}^2(u)\,u_iu_j +\left(\frac{w^2+1}{w^2}\right)\tanh(u)\, u_iw_j \\
&\quad +\frac{\tanh(u)}{w}(w^2-1)u_{ij}.
\end{align*}
Putting these together, we have
\begin{align}\label{eqn:w_ij}
w_{ij} &= \frac{1}{w\cosh^2(u)} \left(u^\ell_j u_{\ell i}+u^\ell D_j u_{\ell i}\right) -\frac{w_j}{w^2\cosh^2(u)}u^\ell u_{\ell i} -\frac{2\tanh(u)}{w\cosh^2(u)}u_j u^\ell u_{\ell i} \notag\\
&\quad -\frac{\text{sech}^2(u)}{w}(w^2-1)u_i u_j -\tanh(u)\left(1+\frac{1}{w^2}\right) u_iw_j -\frac{\tanh(u)}{w}(w^2-1)u_{ij}.
\end{align}

We now evaluate $\mathscr{L}w = \frac{\partial w}{\partial t} - \alpha^{ij}w_{ij} - \beta^k w_k$. Taking derivative of the flow equation \eqref{eqn:flow-eqn}: $u_t = \mathcal{F}(D^2u, Du, u)$ yields $D_k u_t = \alpha^{ij}D_k u_{ij} + \beta^\ell u_{\ell k} + \gamma u_k$, where $\gamma = \frac{\partial \mathcal{F}}{\partial u}$. Expanding $\frac{\partial w}{\partial t}$ using \eqref{eqn:w_t with D_ku_t} gives:
\begin{equation}
\frac{\partial w}{\partial t} = \frac{u^k}{w\cosh^2(u)} (\alpha^{ij}D_k u_{ij} + \beta^\ell u_{\ell k} + \gamma u_k) - \frac{\tanh(u)}{w}(w^2-1)u_t.
\end{equation}
We pair the drift term $\frac{u^k}{w\cosh^2(u)}\beta^\ell u_{\ell k}$ in $\frac{\partial w}{\partial t}$ with the term $-\beta^k w_k$ in $\mathscr{L}w$. By swapping the indices $k$ and $\ell$ and substituting \eqref{eqn:w_k}, we obtain:
\begin{equation}\notag
\frac{u^k}{w\cosh^2(u)}\beta^\ell u_{\ell k} - \beta^k w_k = \beta^k \left( \frac{u^\ell u_{\ell k}}{w\cosh^2(u)} - w_k \right) =  \beta^k \frac{\tanh(u)}{w}(w^2-1)u_k.
\end{equation}

Therefore,
$$
\mathscr{L}w =  \frac{u^k}{w\cosh^2(u)} (\alpha^{ij}D_k u_{ij} + \gamma u_k) - \frac{\tanh(u)}{w}(w^2-1)u_t - \alpha^{ij}w_{ij}+ \beta^k \frac{\tanh(u)}{w}(w^2-1)u_k.
$$

Next, we expand $-\alpha^{ij}w_{ij}$. Using again \eqref{eqn:w_k}, we have
\begin{equation}
u^\ell u_{\ell i} = w\cosh^2(u)w_i + \tanh(u)\cosh^2(u)(w^2-1)u_i.
\end{equation}
Substituting this into the mixed terms of $-\alpha^{ij}w_{ij}$ using \eqref{eqn:w_ij}  yields:
\begin{align*}
&\quad -\alpha^{ij} \left[ -\frac{w_j}{w^2\cosh^2(u)}u^\ell u_{\ell i} - \frac{2\tanh(u)}{w\cosh^2(u)}u_j u^\ell u_{\ell i} \right] \\
&= \frac{1}{w^2\cosh^2(u)} \alpha^{ij} w_j \left[ w\cosh^2(u)w_i + \tanh(u)\cosh^2(u)(w^2-1)u_i \right] \\
&\quad + \frac{2\tanh(u)}{w\cosh^2(u)} \alpha^{ij} u_j \left[ w\cosh^2(u)w_i + \tanh(u)\cosh^2(u)(w^2-1)u_i \right] \\
&= \frac{1}{w}\alpha^{ij}w_i w_j + \frac{\tanh(u)}{w^2}(w^2-1)\alpha^{ij}u_i w_j + 2\tanh(u)\alpha^{ij}u_j w_i + \frac{2\tanh^2(u)}{w}(w^2-1)\alpha^{ij}u_i u_j.
\end{align*}
The expansion of $-\alpha^{ij}w_{ij}$ also contains the term $+ \tanh(u)(1+\frac{1}{w^2})\alpha^{ij}u_i w_j$. Adding this to the above, the term involving $\alpha^{ij}u_i w_j$ becomes:
\begin{equation}\notag
\tanh(u) \left[ \left(1 - \frac{1}{w^2}\right) + 2 + \left(1 + \frac{1}{w^2}\right) \right] \alpha^{ij}u_i w_j = 4\tanh(u)\alpha^{ij}u_i w_j.
\end{equation}

Therefore,
\begin{align*}
\mathscr{L}w =  &\frac{u^k}{w\cosh^2(u)} (\alpha^{ij}D_k u_{ij} + \gamma u_k) - \frac{\tanh(u)}{w}(w^2-1)u_t \\
&- \alpha^{ij}\bigg(\frac{1}{w\cosh^2(u)} \left(u^\ell_j u_{\ell i} + u^\ell D_j u_{\ell i}\right)  -\frac{\text{sech}^2(u)}{w}(w^2-1)u_i u_j \\
&-\frac{\tanh(u)}{w}(w^2-1)u_{ij}  -4\tanh(u) u_i w_j \bigg)\\
&+\frac{1}{w}\alpha^{ij}w_i w_j + \frac{2\tanh^2(u)}{w}(w^2-1)\alpha^{ij}u_i u_j+ \beta^k \frac{\tanh(u)}{w}(w^2-1)u_k.
\end{align*}

We isolate the third-order covariant derivatives remaining in $\mathscr{L}w$, and define
\begin{equation}\notag
\mathcal{C} = \frac{u^k}{w\cosh^2(u)}\alpha^{ij}(D_k u_{ij} - D_j u_{ik}).
\end{equation}
We apply the Ricci identity $D_k u_{ij} - D_j u_{ik} = R_{kji}^p u_p$. For the background hyperbolic metric $(g_0)$ with constant sectional curvature $-1$, the Riemann tensor is $R_{kji}^p = -(\delta_k^p (g_0)_{ij} - \delta_j^p (g_0)_{ki})$. Contracting this with $u_p$ yields $D_k u_{ij} - D_j u_{ik} = u_j (g_0)_{ki} - u_k (g_0)_{ij}$. 
Substituting this into $\mathcal{C}$ and using the definition $\alpha^{ij} = \frac{1}{w^2}g^{ij}$, we evaluate
\begin{align*}
\mathcal{C} &= \frac{1}{w^3 \cosh^2(u)} g^{ij} u^k (u_j (g_0)_{ki} - u_k (g_0)_{ij}) \\
&= \frac{1}{w^3 \cosh^2(u)} g^{ij} (u_i u_j - |Du|^2 (g_0)_{ij}) = - \frac{w^2-1}{w^3 \cosh^2(u)}\,,
\end{align*}
where we used \eqref{eqn:guiuj}: $g^{ij} u_i u_j = \frac{w^2-1}{w^2}$, \eqref{eqn:mismatchedgs}: $g^{ij} (g_0)_{ij} = \frac{2}{\cosh^2(u)} - \frac{|Du|^2}{w^2\cosh^4(u)}$ and $|Du|^2 = \cosh^2(u)(w^2-1)$. Therefore,
\begin{align*}
\mathscr{L}w =  &\frac{\gamma u_k u^k}{w\cosh^2(u)} - \frac{\tanh(u)}{w}(w^2-1)u_t \\
&- \alpha^{ij}\bigg(\frac{1}{w\cosh^2(u)} \left(u^\ell_j u_{\ell i}\right)  -\frac{\text{sech}^2(u)}{w}(w^2-1)u_i u_j -\frac{\tanh(u)}{w}(w^2-1)u_{ij} \\
&-4\tanh(u) u_i w_j\bigg)+\frac{1}{w}\alpha^{ij}w_i w_j + \frac{2\tanh^2(u)}{w}(w^2-1)\alpha^{ij}u_i u_j \\
& - \frac{w^2-1}{w^3 \cosh^2(u)}+ \beta^k \frac{\tanh(u)}{w}(w^2-1)u_k.
\end{align*}

Now to assemble all evaluated components, using \eqref{eqn:u_t}, $u_t = \frac{1}{w^2}g^{ij}u_{ij}-\frac{\tanh(u)}{w^2}\left(3-\frac{1}{w^2}\right)+\frac{\sigma}{w}$ and substituting $\frac{1}{w^2}g^{ij} = \alpha^{ij}$ yields:
\begin{align*}
- \frac{\tanh(u)}{w}(w^2-1)u_t = &- \frac{\tanh(u)}{w}(w^2-1)\alpha^{ij}u_{ij} + \frac{\tanh^2(u)(w^2-1)}{w^3}\left(3-\frac{1}{w^2}\right)\\
&- \frac{\sigma \tanh(u)(w^2-1)}{w^2}.
\end{align*}
Crucially, the term $- \frac{\tanh(u)}{w}(w^2-1)\alpha^{ij}u_{ij}$ cancels the $+ \frac{\tanh(u)}{w}(w^2-1)\alpha^{ij}u_{ij}$ term in $\mathscr{L}w $ above. Therefore,
\begin{align}\label{eqn:lW}
\mathscr{L}w =  &\frac{\gamma u_k u^k}{w\cosh^2(u)} + \frac{\tanh^2(u)(w^2-1)}{w^3}\left(3-\frac{1}{w^2}\right) - \frac{\sigma \tanh(u)(w^2-1)}{w^2}\\
&- \alpha^{ij}\left(\frac{1}{w\cosh^2(u)} \left(u^\ell_j u_{\ell i}\right)  -\frac{\text{sech}^2(u)}{w}(w^2-1)u_i u_j -4\tanh(u) u_iw_j \right)\notag\\
&+\frac{1}{w}\alpha^{ij}w_i w_j + \frac{2\tanh^2(u)}{w}(w^2-1)\alpha^{ij}u_i u_j - \frac{w^2-1}{w^3 \cosh^2(u)}+ \beta^k \frac{\tanh(u)}{w}(w^2-1)u_k.\notag
\end{align}

Now, using \eqref{eqn:gamma} and multiplying this by $\frac{u^k u_k}{w\cosh^2(u)} = \frac{w^2-1}{w}$ we get
\begin{align*}
\frac{\gamma u_ku^k}{w\cosh^2(u)} = &\left(- \frac{2\tanh(u)}{w^4\cosh^2(u)} (g_0)^{ij} u_{ij} + \frac{4\tanh(u)}{w^6\cosh^4(u)} u^i u^j u_{ij}\right)\frac{w^2-1}{w} \\
&-\frac{\text{sech}^2(u)(w^2-1)}{w^3}\left(3-\frac{1}{w^2}\right) -\frac{2\tanh^2(u)(w^2-1)^2}{w^5}\left(3-\frac{2}{w^2}\right) \\
&+ \frac{\sigma \tanh(u)(w^2-1)^2}{w^4}.
\end{align*}

Multiplying $\beta^k$ in \eqref{eqn:beta-k} by $\frac{\tanh(u)}{w}(w^2-1)u_k$ and using $u^k u_k = \cosh^2(u)(w^2-1)$, we have
\begin{align*}
\beta^k \frac{\tanh(u)}{w}(w^2-1)u_k = &- \frac{2\tanh(u)(w^2-1)^2}{w^5\cosh^2(u)} (g_0)^{ij} u_{ij} - \frac{2\tanh(u)(w^2-1)}{w^5\cosh^4(u)} u^i u^j u_{ij} \\
&+ \frac{4\tanh(u)(w^2-1)^2}{w^7\cosh^4(u)} u^i u^j u_{ij} \\
&+ \frac{2\tanh^2(u)(w^2-1)^2}{w^5}\left(3-\frac{2}{w^2}\right) - \frac{\sigma \tanh(u)(w^2-1)^2}{w^4}.
\end{align*}
Note that the $\tanh^2(u)$ and $\sigma$ terms in $\frac{\gamma u_k u^k}{w\cosh^2(u)}$ and $\beta^k \frac{\tanh(u)}{w}(w^2-1)u_k $ cancel out. Now grouping the $(g_0)^{ij} u_{ij}$ terms yields:
\begin{equation}\notag
- \frac{2\tanh(u)(w^2-1)}{w^5\cosh^2(u)} \big[ 1 + (w^2-1) \big] (g_0)^{ij} u_{ij} = - \frac{2\tanh(u)(w^2-1)}{w^3\cosh^2(u)} (g_0)^{ij} u_{ij}.
\end{equation}
Grouping the $u^i u^j u_{ij}$ terms yields:
\begin{equation}\notag
\frac{2\tanh(u)(w^2-1)}{w^7\cosh^4(u)} \big[ 2 - w^2 + 2(w^2-1) \big] u^i u^j u_{ij} = \frac{2\tanh(u)(w^2-1)}{w^5\cosh^4(u)} u^i u^j u_{ij}.
\end{equation}
Therefore,
\begin{align}\label{eqn:gamma-beta-k}
    &\frac{\gamma u_k u^k}{w\cosh^2(u)}+ \beta^k \frac{\tanh(u)}{w}(w^2-1)u_k \\
    =&  - \frac{2\tanh(u)(w^2-1)}{w^3\cosh^2(u)} (g_0)^{ij} u_{ij} + \frac{2\tanh(u)(w^2-1)}{w^5\cosh^4(u)} u^i u^j u_{ij} -\frac{\text{sech}^2(u)(w^2-1)}{w^3}\left(3-\frac{1}{w^2}\right)\notag\\
    =& - \frac{2\tanh(u)(w^2-1)}{w}\left(\frac{(g_0)^{ij}}{w^2\cosh^2(u)} - \frac{u^i u^j}{w^4\cosh^4(u)}\right)u_{ij}-\frac{\text{sech}^2(u)(w^2-1)}{w^3}\left(3-\frac{1}{w^2}\right)\notag\\
    =& - \frac{2\tanh(u)(w^2-1)}{w} \alpha^{ij} u_{ij}-\frac{\text{sech}^2(u)(w^2-1)}{w^3}\left(3-\frac{1}{w^2}\right)\,,\notag
\end{align}
where we have used \eqref{eqn:alpha-ij} for $\alpha^{ij}$.

Finally, we substitute \eqref{eqn:gamma-beta-k} back into $\mathscr{L}w$ in \eqref{eqn:lW}. Using \eqref{eqn:alpha-ij}, we have 
\begin{equation}\label{eqn:alpha-u-i-u-j-1}
\alpha^{ij}u_i u_j =  \frac{1}{w^2 \cosh^2(u)} \left( |Du|^2 - \frac{|Du|^4}{w^2 \cosh^2(u)} \right) =\frac{w^2-1}{w^4}.
\end{equation}
Grouping all $\text{sech}^2(u)$ terms (using $\text{sech}^2(u) = \frac{1}{\cosh^2(u)}$) in \eqref{eqn:lW}:
\begin{align*}
\Sigma_{\text{sech}} &= \frac{\text{sech}^2(u)(w^2-1)}{w^5} \left[ -w^2\left(3-\frac{1}{w^2}\right) + (w^2-1) - w^2 \right] \\
& = - \frac{3(w^2-1)}{w^3\cosh^2(u)}.
\end{align*}
Grouping all $\tanh^2(u)$ terms in \eqref{eqn:lW}:
\begin{align*}
\Sigma_{\tanh} &= \frac{\tanh^2(u)(w^2-1)}{w^5} \left[ w^2\left(3-\frac{1}{w^2}\right) + 2(w^2-1) \right] \\
& = \frac{\tanh^2(u)(w^2-1)(5w^2-3)}{w^5}.
\end{align*}

Finally, we have
\begin{align} \label{eqn:L_w_exact_identity}
\mathscr{L}w = & - \frac{1}{w\cosh^2(u)} \alpha^{ij}(g_0)^{k\ell} u_{ik} u_{j\ell} - \frac{2\tanh(u)(w^2-1)}{w} \alpha^{ij} u_{ij} \nonumber \\
& - \frac{3(w^2-1)}{w^3\cosh^2(u)}  + \frac{1}{w} \alpha^{ij} w_i w_j + 4\tanh(u) \alpha^{ij} u_i w_j \notag\\
&+ \frac{\tanh^2(u)(w^2-1)(5w^2-3)}{w^5} - \frac{\sigma \tanh(u)(w^2-1)}{w^2}.
\end{align}

We bound the right-hand side of \eqref{eqn:L_w_exact_identity} term by term except the terms involving $w_j$. First, we consider the sum of the following three terms:
\begin{equation}\notag
P(w) = - \frac{3(w^2-1)}{w^3\cosh^2(u)} + \frac{\tanh^2(u)(w^2-1)(5w^2-3)}{w^5} - \frac{\sigma \tanh(u)(w^2-1)}{w^2}.
\end{equation}
Since $w \ge 1$ and the height function $u$ is uniformly bounded ($\|u\|_{C^0} \le C_{b}$) by Lemma \ref{lem:height_bound}, the first two terms of $P(w)$ decay as $\mathcal{O}(w^{-1})$ as $w \to \infty$, and the third term converges to a constant bounded by $|\sigma| \tanh(C_b)$ as $w \to \infty$. Therefore, there exists a uniform constant $C_1 > 0$ such that $P(w) \le C_1 \le C_1 w$.

Now we define the primary positive semi-definite quadratic Hessian term in \eqref{eqn:L_w_exact_identity} as 
\begin{equation}\label{eqn:mathcal-Q}
\mathcal{Q} = \frac{1}{w\cosh^2(u)} \alpha^{ij} (g_0)^{k\ell} u_{ik} u_{j\ell}.
\end{equation}
To see $\mathcal{Q}$ is positive semi-definite, we first note that using \eqref{eqn:alpha-ij}: $\alpha^{ij} = \frac{(g_0)^{ij}}{w^2\cosh^2(u)} - \frac{u^i u^j}{w^4\cosh^4(u)}$ and substituting $|Du|_{g_0}^2 = \cosh^2(u)(w^2-1)$, we have
\begin{align}\label{eqn:alpha-u-i}\notag
\alpha^{ij}u_i &= \frac{u^j}{w^2\cosh^2(u)} - \frac{|Du|_{g_0}^2 u^j}{w^4\cosh^4(u)} \nonumber \\
&= \frac{u^j}{w^2\cosh^2(u)}\left(1 - \frac{w^2-1}{w^2}\right) = \frac{u^j}{w^4\cosh^2(u)}.
\end{align}
Thus, the gradient vector $Du$ is an eigenvector of $\alpha^{ij}$ with respect to $g_0$ with eigenvalue:
\begin{equation}\label{eqn:Du-as-eigenvector}
\lambda_{\parallel} = \frac{1}{w^4\cosh^2(u)}.
\end{equation}
Now let $V$ be a tangent vector such that $\langle V, Du \rangle_{g_0} = u_i V^i = 0$. Contracting $\alpha^{ij}$ with $V_j$ yields:
\begin{align*}
\alpha^{ij} V_j &= \frac{1}{w^2\cosh^2(u)} (g_0)^{ij} V_j - \frac{1}{w^4\cosh^4(u)} u^i (u^j V_j) \\
&= \frac{1}{w^2\cosh^2(u)} V^i.
\end{align*}
Thus, any vector orthogonal to $Du$ is an eigenvector of $\alpha^{ij}$ with eigenvalue:
\begin{equation}\label{eqn:max-eigenvalue}
\lambda_{\perp} = \frac{1}{w^2\cosh^2(u)}.
\end{equation}
Therefore, the minimum and maximum eigenvalues of $\alpha^{ij}$ are
\begin{equation}
\lambda_{\min} = \lambda_{\parallel} = \frac{1}{w^4\cosh^2(u)} \quad \text{and} \quad \lambda_{\max} = \lambda_{\perp} = \frac{1}{w^2\cosh^2(u)}\,.
\end{equation}
By definition, $w\cosh^2(u)\,\mathcal{Q}$ is the trace of a positive semi-definite 2-tensor $M = M_{ij} = (g_0)^{k\ell} u_{ik} u_{j\ell}$ against $\alpha^{ij}$. To see $M$ is positive semi-definite, we contract it with an arbitrary tangent vector $V$:
\begin{equation}
M_{ij} V^i V^j = (g_0)^{k\ell} (u_{ik} V^i) (u_{j\ell} V^j) = (g_0)^{k\ell} W_k W_\ell = |W|_{g_0}^2 \ge 0\,,
\end{equation}
where $W_k = u_{ik} V^i$.
Therefore, because $M_{ij}$ is positive semi-definite, its trace against the positive semi-definite metric $\alpha^{ij}$ is bounded below by the product of the minimum eigenvalue of $\alpha$ and the $g_0$-trace of $M$:
\begin{align} \label{eqn:mathcal-Q}
\mathcal{Q} =& \frac{1}{w\cosh^2(u)}\alpha^{ij} M_{ij} \ge \frac{1}{w\cosh^2(u)} \lambda_{\min}(\alpha) \text{tr}_{g_0}(M) \\
=& \frac{1}{w\cosh^2(u)} \lambda_{\min}(\alpha)  (g_0)^{ij} M_{ij} = \frac{1}{w\cosh^2(u)} \lambda_{\min}(\alpha)  (g_0)^{ij} (g_0)^{k\ell} u_{ik} u_{j\ell} \notag\\
=& \frac{1}{w^5\cosh^4(u)}  |D^2u|_{g_0}^2 \geq 0.\notag
\end{align}

Now we evaluate the remaining term $\mathcal{L}_{trace} = - \frac{2\tanh(u)(w^2-1)}{w} \alpha^{ij} u_{ij}$ in \eqref{eqn:L_w_exact_identity}. Because $\alpha^{ij}$ is positive-definite and symmetric, let $S^{ij}$ be the unique positive-definite symmetric square root of $\alpha^{ij}$ such that $\alpha^{ij} = \sum_{k=1}^2 S^{ik}S^{kj}$. We define the $2\times 2$ matrix $M = M^{k \ell} = S^{ik} u_{ij} S^{j\ell}$ and rewrite 
\begin{equation}
\alpha^{ij} u_{ij} =  \sum_{k=1}^2 S^{ik} S^{kj} u_{ij} =  \sum_{k=1}^2 S^{ik} u_{ij} S^{j\ell} \delta^k_{\ell} = \sum_{k=1}^2 M^{k \ell} \delta^k_{\ell} = \text{tr}(M).
\end{equation}
We apply the trace inequality
\begin{equation}
(\text{tr}(M))^2 \leq 2 \text{tr}(M^2),
\end{equation}
which yields
\begin{align}\label{eqn:alpha-ij-to-M}
(\alpha^{ij}u_{ij})^2 \leq& 2\text{tr}(M^2) = 2\sum_{\ell=1}^2 \sum_{k=1}^2 (S^{ik} u_{ij} S^{j\ell}) (S^{p\ell} u_{pq} S^{qk}) \notag\\
=& 2\alpha^{qi} u_{ij} \alpha^{jp} u_{pq} = 2\alpha^{ik} \alpha^{j\ell} u_{ij} u_{k\ell},
\end{align}
where we used $S^{j\ell} = S^{\ell j}$, $\sum_{\ell=1}^2 S^{j\ell} S^{p\ell} = \alpha^{jp}$ and $\sum_{k=1}^2 S^{qk} S^{ik} = \alpha^{qi}$. Because the max eigenvalue of $\alpha^{ij}$ is $\frac{1}{w^2\cosh^2(u)}$ (see \eqref{eqn:max-eigenvalue}), by \eqref{eqn:alpha-ij-to-M} and the definition of $\mathcal{Q}$ we have 
\begin{equation}\label{eqn:alpha-u-ij-sq}
(\alpha^{ij}u_{ij})^2 \le \frac{2}{w^2\cosh^2(u)} \alpha^{ik}(g_0)^{j\ell} u_{ij}u_{k\ell} = \frac{2}{w}\mathcal{Q}.
\end{equation}
Therefore,
\begin{equation}
|\mathcal{L}_{trace}| \le \frac{2(w^2-1)}{w} |\alpha^{ij}u_{ij}| \le 2w \sqrt{\frac{2}{w}\mathcal{Q}} = 2\sqrt{2w}\sqrt{\mathcal{Q}} \le  \frac{\mathcal{Q}}{2} + 4 w\,,
\end{equation}
where we applied Young's inequality in the last step. Combining these estimates we obtain
\begin{equation}\label{eqn:w_bound_linear_global}
\mathscr{L} w \le C_0 w - \frac{\mathcal{Q}}{2}  + \frac{1}{w} \alpha^{ij} w_i w_j + 4\tanh(u) \alpha^{ij} u_i w_j.
\end{equation}
Discarding the negative-definite term $- \frac{\mathcal{Q}}{2}$ yields $\mathscr{L}w \le C_0 w + \frac{1}{w} \alpha^{ij} w_i w_j + 4\tanh(u) \alpha^{ij} u_i w_j$.
\end{proof}

\begin{corollary}[Time-Dependent Exponential Gradient Bound]\label{cor:exponential_bound}
Let $u \in C^{2,1}(\Sigma \times [0,T))$ be a smooth solution to the modified mean curvature flow \eqref{eqn:u_t}. The gradient function $w$ satisfies the time-dependent exponential bound:
\begin{equation}\label{eqn:exp-grad-bound}\notag
w(x,t) \le \left( \sup_{x \in \Sigma} w(x,0) \right) e^{C_0 t} \qquad \forall (x,t) \in \Sigma \times [0,T),
\end{equation}
where the constant $C_0 > 0$ is given in Lemma \ref{lem:w_evolution}. Consequently, the flow cannot develop finite-time gradient singularities, ensuring that the smooth solution can be extended for all time $t \in [0,\infty)$. Moreover, the flow converges exponentially in $C^0$ to the equidistant surface $\Sigma(r_\sigma)$ where $r_\sigma = \tanh^{-1}\left(\frac{\sigma}{2}\right)$ is from Lemma \ref{lem:height_bound}.
\end{corollary}

\begin{proof}
We define the vector field $\tilde{\beta}^k$:
\begin{equation}\notag
\tilde{\beta}^k = \beta^k + \frac{1}{w}\alpha^{ik}w_i + 4\tanh(u)\alpha^{ik}u_i
\end{equation}
and the corresponding linear parabolic operator 
\begin{equation}\label{eqn:tilde-L}\notag
\tilde{\mathscr{L}} = \frac{\partial}{\partial t} - \alpha^{ij} D_i D_j - \tilde{\beta}^k D_k.
\end{equation}
The inequality \eqref{eqn:w_bound_linear} can be rewritten as:
\begin{equation}\notag
\tilde{\mathscr{L}} w \le C_0 w.
\end{equation}

Now, let $\tilde{w}(x,t) = e^{-C_0 t} w(x,t)$. We have
\begin{align*}
\tilde{\mathscr{L}} \tilde{w} &= \frac{\partial}{\partial t}(e^{-C_0 t} w) - \alpha^{ij} D_i D_j (e^{-C_0 t} w) - \tilde{\beta}^k D_k (e^{-C_0 t} w) \\
&= -C_0 e^{-C_0 t} w + e^{-C_0 t} \frac{\partial w}{\partial t} - e^{-C_0 t} \alpha^{ij} w_{ij} - e^{-C_0 t} \tilde{\beta}^k w_k \\
&= e^{-C_0 t} (\tilde{\mathscr{L}} w - C_0 w) \le 0. 
\end{align*} 
Because $\tilde{\mathscr{L}}$ is a uniformly parabolic operator with continuous coefficients, the standard parabolic maximum principle guarantees that the spatial maximum of the sub-solution $\tilde{w}$ is non-increasing in time. Therefore, for all $t \in [0,T)$, we have:
\begin{equation}\notag
e^{-C_0 t} w(x,t) = \tilde{w}(x,t) \le \sup_{x \in \Sigma} \tilde{w}(x,0) = \sup_{x \in \Sigma} w(x,0)\,,
\end{equation}
which yields \eqref{eqn:exp-grad-bound} by multiplying both sides by $e^{C_0 t}$.
\end{proof}

\section{Uniform Gradient Bound and Convergence}\label{sec:unif-gra-boud-conv}
\subsection{The case $\sigma=0$}
We now use the test function 
$$\Phi(x,t) = \cosh(u(x,t)) w(x,t)$$ 
to derive an improved linear bound on the gradient function $w$, and in particular a uniform gradient bound in the case of mean curvature flow ($\sigma=0$).

\begin{theorem}\label{thm:gradient_bound-sigma-0}
Let $u \in C^{2,1}(\Sigma \times [0,T))$ be a smooth solution to the modified mean curvature flow \eqref{eqn:u_t} with $\sigma\in (-2,2)$. There exist constants $C_1, C_2 > 0$, depending only on the background metric $g_0$ on $\Sigma$, the $C^0$ bound of $u$ in Lemma \ref{lem:height_bound}, the constant $\sigma$, and the initial gradient bound $\|w(\cdot, 0)\|_\infty$, such that:
\begin{equation}
w(x,t) \le C_1 (1+t) \qquad \forall (x,t) \in \Sigma \times [0,T).
\end{equation}
Moreover, when $\sigma=0$ we have
\begin{equation}
w(x,t) \le C_2 \qquad \forall (x,t) \in \Sigma \times [0,T).
\end{equation}
\end{theorem}

\begin{proof}
We compute the linear parabolic operator $\mathscr{L} = \partial_t - \alpha^{ij} D_i D_j - \beta^k D_k$ from \eqref{eqn:operator-scr-L} acting on the test function $\Phi = \cosh(u)w$. To do so, we first compute:
\begin{align*}
\Phi_k &= \sinh(u) u_k w + \cosh(u) w_k, \\
\Phi_t &= \sinh(u) u_t w + \cosh(u) w_t, \\
\Phi_{ij} &= \cosh(u) u_i u_j w + \sinh(u) u_{ij} w + \sinh(u) (u_i w_j + u_j w_i) + \cosh(u) w_{ij}.
\end{align*}
Substituting these into $\mathscr{L}$ yields
\begin{equation}\label{eqn:L_Phi_expansion_cosh}
\mathscr{L}\Phi = \cosh(u) \mathscr{L}w + w\sinh(u) \mathscr{L}u - w\cosh(u) \alpha^{ij} u_i u_j - 2\sinh(u) \alpha^{ij} u_i w_j.
\end{equation}

Define the spatial maximum function $\Phi_{\max}(t) = \sup_{x \in \Sigma} \Phi(x,t)$. By the compactness of the closed 2-surface $\Sigma$, for each $t \in [0,T)$, there exists a point $x_t \in \Sigma$ such that $\Phi(x_t, t) = \Phi_{\max}(t)$. At this spatial maximum point $(x_t,t)$, the gradient vanishes ($\Phi_k(x_t, t) = 0$), which implies that at $(x_t,t)$ we have
\begin{equation}\label{eqn:w_i_max_relation_cosh}
w_k= -\tanh(u) w u_k.
\end{equation}
Furthermore, the Hessian matrix of $\Phi$ is negative semi-definite ($\Phi_{ij} (x_t, t)\le 0$). Because $\alpha^{ij} = \frac{1}{w^2}g^{ij}$ is strictly positive-definite, we have $-\alpha^{ij} \Phi_{ij} (x_t, t)\ge 0$. Thus, using Hamilton's trick we have
\begin{equation}\label{eqn:L_Phi_max_cosh}\notag
\frac{d}{dt} \Phi_{\max}(t) \le \mathscr{L}\Phi(x_t, t).
\end{equation}

Substituting \eqref{eqn:w_i_max_relation_cosh} into the cross term in \eqref{eqn:L_Phi_expansion_cosh}, we have  
\[- 2\sinh(u) \alpha^{ij} u_i w_j = -2\sinh(u)\alpha^{ij}u_i(-\tanh(u)w u_j) = 2w \frac{\sinh^2(u)}{\cosh(u)} \alpha^{ij}u_i u_j\,,\] 
Combining this with the $u_i u_j$ term in \eqref{eqn:L_Phi_expansion_cosh}, we obtain:
\begin{equation}\label{eqn:LPhi}
\mathscr{L}\Phi = \cosh(u) \mathscr{L}w + w\sinh(u) \mathscr{L}u + w \left( \frac{2\sinh^2(u) - \cosh^2(u)}{\cosh(u)} \right) \alpha^{ij} u_i u_j.
\end{equation}

Next, we evaluate $\mathscr{L}u = u_t - \alpha^{ij}u_{ij} - \beta^k u_k$. Using the flow equation \eqref{eqn:u_t}, we have 
\begin{equation}\label{eqn:flow-eqn-original}
u_t - \alpha^{ij}u_{ij} = - \frac{\tanh(u)}{w^2}\left(3-\frac{1}{w^2}\right) + \frac{\sigma}{w}.
\end{equation}
To compute $\beta^k u_k$, we contract $\beta^k$ with $u_k$ and use $|Du|^2 = u^k u_k = \cosh^2(u)(w^2-1)$ to get:
\begin{align}\label{eqn:beta-k-u-k-original}
\beta^k u_k = & \bigg(- \frac{2u^k}{w^4\cosh^4(u)} (g_0)^{ij} u_{ij} - \frac{2(g_0)^{ik} u^j u_{ij}}{w^4\cosh^4(u)} + \frac{4u^k u^i u^j u_{ij}}{w^6\cosh^6(u)} \notag\\
& + \frac{2\tanh(u)u^k}{w^4\cosh^2(u)}\left(3 - \frac{2}{w^2}\right) - \frac{\sigma u^k}{w^3\cosh^2(u)}\bigg) u_k \notag\\
=&  - \frac{2(w^2-1)}{w^4\cosh^2(u)} (g_0)^{ij} u_{ij} - \frac{2 u^i u^j u_{ij}}{w^4\cosh^4(u)} + \frac{4(w^2-1) u^i u^j u_{ij}}{w^6\cosh^4(u)} + P_\beta \notag\\
= &- \frac{2(w^2-1)}{w^2} \left[ \frac{(g_0)^{ij}}{w^2\cosh^2(u)} - \frac{u^i u^j}{w^4\cosh^4(u)}  \right]u_{ij} - \frac{2- 2(w^2-1)/w^2}{w^4\cosh^4(u)} u^i u^j u_{ij} + P_\beta \notag\\
= & - \frac{2(w^2-1)}{w^2} \alpha^{ij}u_{ij} - \frac{2}{w^6\cosh^4(u)} u^i u^j u_{ij} + P_\beta,
\end{align}
where we used \eqref{eqn:alpha-ij} to get $\alpha^{ij}u_{ij} = \frac{(g_0)^{ij}}{w^2\cosh^2(u)}u_{ij} - \frac{u^i u^j}{w^4\cosh^4(u)}u_{ij}$ and
\begin{equation}
P_\beta = \frac{2\tanh(u)(w^2-1)}{w^4}\left(3-\frac{2}{w^2}\right) - \frac{\sigma(w^2-1)}{w^3} = \frac{6\tanh(u)}{w^2} - \frac{\sigma}{w} + \mathcal{O}(w^{-3}),
\end{equation}
as $w \to \infty$.

Subtracting \eqref{eqn:beta-k-u-k-original} from \eqref{eqn:flow-eqn-original} yields:
\begin{equation}\label{eqn:L_u_evaluated}
\mathscr{L}u = \frac{2(w^2-1)}{w^2} \alpha^{ij}u_{ij} + \frac{2}{w^6\cosh^4(u)} u^i u^j u_{ij} + P_u,
\end{equation}
where 
\[ P_u = - P_\beta - \frac{\tanh(u)}{w^2}\left(3-\frac{1}{w^2}\right) + \frac{\sigma}{w} = - \frac{9\tanh(u)}{w^2} + \frac{2\sigma}{w} + \mathcal{O}(w^{-3}) \quad \text{as } w\to \infty. \] 

Next, substituting \eqref{eqn:w_i_max_relation_cosh} into the formula for $w_k$ from \eqref{eqn:w_k}, we have
\begin{equation*}
w_k =\frac{u^\ell u_{\ell k}}{w\cosh^2(u)} - \frac{\tanh(u)}{w}(w^2-1)u_k = -\tanh(u)w u_k \implies u^\ell u_{\ell k} = -\tanh(u)\cosh^2(u) u_k.
\end{equation*}
Contracting this with $u^k$ yields that at the maximum point $(x_t, t)$, we have
\begin{equation*}
u^i u^j u_{ij} = u^k(u^\ell u_{\ell k}) = -\tanh(u)\cosh^2(u)|Du|^2 = -\tanh(u)\cosh^4(u)(w^2-1).
\end{equation*}
Substituting this into $w\sinh(u) \mathscr{L}u$ using \eqref{eqn:L_u_evaluated}, we get
\begin{equation}\label{eqn:cross-term-in-Lu}
w\sinh(u) \left[ \frac{2}{w^6\cosh^4(u)} u^i u^j u_{ij} \right] = - \frac{2\sinh^2(u)(w^2-1)}{w^5\cosh(u)} = \mathcal{O}(w^{-3}) \quad{\text{as }} w\to \infty.
\end{equation}

Recall that we have the identity \eqref{eqn:L_w_exact_identity} for $\mathscr{L}w$:
\begin{align}\label{eqn:L_w_exact_identity-new}
\mathscr{L}w = & - \frac{1}{w\cosh^2(u)} \alpha^{ij}(g_0)^{k\ell} u_{ik} u_{j\ell} - \frac{2\tanh(u)(w^2-1)}{w} \alpha^{ij} u_{ij} \nonumber \\
& - \frac{3(w^2-1)}{w^3\cosh^2(u)} + \frac{1}{w} \alpha^{ij} w_i w_j + 4\tanh(u) \alpha^{ij} u_i w_j \notag\\
& + \frac{\tanh^2(u)(w^2-1)(5w^2-3)}{w^5} - \frac{\sigma \tanh(u)(w^2-1)}{w^2}.
\end{align}
At the spatial maximum point $(x_t, t)$, we substitute \eqref{eqn:w_i_max_relation_cosh} into the gradient cross-terms of \eqref{eqn:L_w_exact_identity-new}:
\begin{align*}
\frac{1}{w} \alpha^{ij} w_i w_j &= \frac{1}{w} \alpha^{ij} \big( -\tanh(u) w u_i \big) \big( -\tanh(u) w u_j \big) = w\tanh^2(u) \alpha^{ij}u_i u_j, \\
4\tanh(u) \alpha^{ij} u_i w_j &= 4\tanh(u) \alpha^{ij} u_i \big( -\tanh(u) w u_j \big) = -4w\tanh^2(u) \alpha^{ij}u_i u_j.
\end{align*}
Summing these evaluates to $-3w\tanh^2(u) \alpha^{ij}u_i u_j$. Applying \eqref{eqn:alpha-u-i-u-j-1}: $\alpha^{ij}u_i u_j = \frac{w^2-1}{w^4}$, this reduces to:
\begin{equation*}
-3w\tanh^2(u) \left(\frac{w^2-1}{w^4}\right) = - \frac{3\tanh^2(u)(w^2-1)}{w^3}.
\end{equation*}
Combining this with the $\frac{\tanh^2(u)(w^2-1)(5w^2-3)}{w^5}$ term from \eqref{eqn:L_w_exact_identity-new} yields:
\begin{equation*}
- \frac{3\tanh^2(u)(w^2-1)}{w^3} + \frac{\tanh^2(u)(w^2-1)(5w^2-3)}{w^5} = \frac{\tanh^2(u)(w^2-1)(2w^2-3)}{w^5}.
\end{equation*}
Therefore, evaluated at the maximum point $(x_t, t)$, $\mathscr{L}w$ simplifies to:
\begin{align}\label{eqn:L_w_simplified}
\mathscr{L}w = & - \frac{1}{w\cosh^2(u)} \alpha^{ij}(g_0)^{k\ell} u_{ik} u_{j\ell} - \frac{2\tanh(u)(w^2-1)}{w} \alpha^{ij}u_{ij}  \notag\\
& - \frac{3(w^2-1)}{w^3\cosh^2(u)} + \frac{\tanh^2(u)(w^2-1)(2w^2-3)}{w^5} - \frac{\sigma \tanh(u)(w^2-1)}{w^2}.
\end{align}

Now we substitute \eqref{eqn:L_u_evaluated} and \eqref{eqn:L_w_simplified} into \eqref{eqn:LPhi}. First, we consider the two Hessian trace terms involving $\alpha^{ij}u_{ij}$:
\begin{align*}
\text{Trace}_{\mathscr{L}w} &:= \cosh(u) \left[ - \frac{2\tanh(u)(w^2-1)}{w} \alpha^{ij}u_{ij} \right] = - \frac{2\sinh(u)(w^2-1)}{w} \alpha^{ij}u_{ij}, \\
\text{Trace}_{\mathscr{L}u} &:= w\sinh(u) \left[ \frac{2(w^2-1)}{w^2} \alpha^{ij}u_{ij} \right] =  \frac{2\sinh(u)(w^2-1)}{w} \alpha^{ij}u_{ij}.
\end{align*}
Therefore, these two trace terms sum to zero at the maximum point $(x_t,t)$ in \eqref{eqn:LPhi}.

Combining these using \eqref{eqn:L_u_evaluated}, \eqref{eqn:cross-term-in-Lu}, \eqref{eqn:L_w_simplified}, and \eqref{eqn:alpha-u-i-u-j-1}, we obtain:
\begin{align}\label{eqn:LPhi_updated}
\mathscr{L}\Phi = &\cosh(u) \mathscr{L}w + w\sinh(u) \mathscr{L}u + w \left( \frac{2\sinh^2(u) - \cosh^2(u)}{\cosh(u)} \right) \alpha^{ij} u_i u_j \notag\\
=&\cosh(u) \Big( - \frac{1}{w\cosh^2(u)} \alpha^{ij}(g_0)^{k\ell} u_{ik} u_{j\ell} - \frac{3(w^2-1)}{w^3\cosh^2(u)} \notag\\
& + \frac{\tanh^2(u)(w^2-1)(2w^2-3)}{w^5} - \frac{\sigma \tanh(u)(w^2-1)}{w^2}\Big) \notag\\
&+ w\sinh(u) \Big( -\frac{9\tanh(u)}{w^2} + \frac{2\sigma}{w} + \mathcal{O}(w^{-3}) \Big) + w \left( \frac{2\sinh^2(u) - \cosh^2(u)}{\cosh(u)} \right) \frac{w^2-1}{w^4} \notag\\
=& - \cosh(u) \mathcal{Q} + \sigma \sinh(u) - \frac{6\sinh^2(u) + 4}{w\cosh(u)} + \mathcal{O}(w^{-2})\,,
\end{align}
where $\mathcal{Q}$ is the positive semi-definite quadratic Hessian from \eqref{eqn:mathcal-Q}.

Note that the primary obstruction here is the dominating term $\sigma\sinh(u)$ when $w$ is large. For $\sigma=0$, we have:
\begin{equation}
\frac{d}{dt} \Phi_{\max}(t) \le \mathscr{L}\Phi(x_t, t) \le - \frac{6\sinh^2(u) + 4}{w\cosh(u)} + \mathcal{O}(w^{-2}).
\end{equation}
By Lemma \ref{lem:height_bound}, $\cosh(u)$ is uniformly bounded from above (and trivially bounded from below by $1$). Therefore, there exists a sufficiently large constant $\bar{C} > 0$ such that whenever $\Phi_{\max}(t) > \bar{C}$, we have $\frac{d}{dt} \Phi_{\max}(t) < 0$. Arguing as in the proof of Lemma \ref{lem:height_bound}, we have
\begin{equation}
\Phi(x,t) \le \max \left\{ \sup_{x \in \Sigma} \Phi(x,0), \bar{C} \right\} := C_\Phi,
\end{equation}
which, since $w \le \Phi$, immediately implies $w(x,t) \le C_2$.

For general $\sigma\in (-2,2)$, Lemma \ref{lem:height_bound} again guarantees that $\sigma\sinh(u)$ is uniformly bounded, and thus $\frac{d}{dt} \Phi_{\max}(t) \le \tilde{C}$. Integrating this from $t=0$ yields:
\begin{equation}\label{eqn:linear_gradient_bound}\notag
w(x,t) \le \Phi(x,t) \le \Phi_{\max}(0) + \tilde{C} t \le C_1(1 + t).\quad \quad \qed \qedhere
\end{equation}
\end{proof}

\begin{remark}
    It will be interesting to find an appropriate test function to prove the uniform gradient estimate for general $\sigma\in (-2,2)$.
\end{remark}

\subsection{The general case $\sigma \in (-2,2)$} We first recall the evolution equations of the gradient function $w$ and the squared norm of the second fundamental form $|A|^2$ under the modified mean curvature flow.

\begin{lemma}\label{lem:evol_w}
Let $u \in C^{2,1}(\Sigma \times [0,T))$ be a smooth solution to the modified mean curvature flow \eqref{eqn:u_t} with $\sigma\in (-2,2)$. Then the gradient function $w$ satisfies the evolution equation:
\begin{align}\label{eqn:evol_w_modified}
(\partial_t - \Delta) w &= -w|A|^2 + 2w^2\tanh(u)H + 2\tanh(u)\langle \nabla w, \mathbf{n} \rangle \notag\\
&\quad  - \frac{w^2-1}{w\cosh^2(u)} - 2w\tanh^2(u) - \frac{2}{w}|\nabla w|^2  - \sigma\tanh(u)(w^2-1)\,,
\end{align}
where $A$ is the second fundamental form of the evolving surface $\Sigma_t$ and $\mathbf{n}= \frac{\partial}{\partial r}$.
\end{lemma}

\begin{proof}
By \cite[Theorem 3.6]{HLZ2020} and its proof, the evolution equation for $\Theta = 1/w$ under the modified mean curvature flow is given by:
\begin{align}\label{eqn:evo-Theta}
(\partial_t - \Delta)\Theta &= |A|^2\Theta - 2\tanh(u)H + 2\tanh(u)\langle\nabla\Theta, \mathbf{n}\rangle \notag\\
&\quad + \frac{\Theta(1-\Theta^2)}{\cosh^2(u)} + 2\tanh^2(u)\Theta + \sigma\tanh(u)(1 - \Theta^2).
\end{align}
Since $w = \Theta^{-1}$, we have $\nabla w = -w^2 \nabla \Theta$ and:
\begin{equation}\notag
(\partial_t - \Delta)w = -w^2(\partial_t - \Delta)\Theta - \frac{2}{w}|\nabla w|^2.
\end{equation}
Multiplying the right-hand side of \eqref{eqn:evo-Theta} by $-w^2$ term-by-term, we obtain:
\begin{align*}
-w^2\left(|A|^2 w^{-1}\right) &= -w|A|^2, \\
-w^2\left(-2\tanh(u)H\right) &= 2w^2\tanh(u)H, \\
-w^2\left(2\tanh(u)\langle -w^{-2}\nabla w, \mathbf{n} \rangle\right) &= 2\tanh(u)\langle \nabla w, \mathbf{n} \rangle, \\
-w^2\left(\frac{w^{-1}(1-w^{-2})}{\cosh^2(u)}\right) &= -\frac{w^2-1}{w\cosh^2(u)}, \\
-w^2\left(2\tanh^2(u)w^{-1}\right) &= -2w\tanh^2(u), \\
-w^2\left(\sigma\tanh(u)(1-w^{-2})\right) &= -\sigma\tanh(u)(w^2-1)\,,
\end{align*}
which yields \eqref{eqn:evol_w_modified}.
\end{proof}

\begin{lemma}\label{lem:evol_A}
Let $u \in C^{2,1}(\Sigma \times [0,T))$ be a smooth solution to the modified mean curvature flow \eqref{eqn:u_t} with $\sigma\in (-2,2)$. Then the squared norm of the second fundamental form $|A|^2$ of the evolving surface $\Sigma_t$ satisfies the following evolution equation:
\begin{align}
\left(\frac{\partial}{\partial t} - \Delta\right) |A|^2 &= -2|\nabla A|^2 + 2|A|^4 + 4(|A|^2 - H^2) -2\sigma \text{tr}(A^3) - 2\sigma H. \label{eqn:evol_A2_mod}
\end{align}
\end{lemma}

\begin{proof}
By \cite[Theorem 2.2]{HLZ2020}, the evolution equation for $|A|^2$ under the standard mean curvature flow is given by:
\begin{align}\label{eqn:evo-2nd-FF-MCF}
\left(\frac{\partial}{\partial t} - \Delta\right) |A|^2 &= -2|\nabla A|^2 + 2|A|^4 + 4(|A|^2 - H^2).
\end{align}

Following Huisken's \cite[Lemma 3.3 and Theorem 3.4]{H86}, under a normal variation $\partial_t F = f\nu$ for a smooth function $f$, the induced metric $g_{ij}$, its inverse $g^{ij}$, and the components of the second fundamental form of the evolving surface satisfy the evolution equations:
\begin{align*}
\frac{\partial}{\partial t} g_{ij} &= 2f h_{ij}, \\
\frac{\partial}{\partial t} g^{ij} &= -2f h^{ij},\\
\frac{\partial}{\partial t} h_{ij} &= -\nabla_i \nabla_j f + f h_{il}h^l_j + f\overline{R}_{0i0j},
\end{align*}
where $\overline{R}_{0i0j} = -\langle \bar{R}(\nu, \partial_i F)\partial_j F, \nu \rangle$ denotes the ambient Riemann curvature tensor evaluated in the normal direction.

Compared with \eqref{eqn:evo-2nd-FF-MCF}, the additional contributions to the evolution equation of $|A|^2$ under the modified mean curvature flow arise from $\frac{\partial}{\partial t}$ acting on $|A|^2$. Because the time derivative operator is linear with respect to the normal variation speed, the evolution under the modified mean curvature flow $\partial_t F = -(H-\sigma)\nu$ decomposes into the standard mean curvature flow contribution ($f = -H$) and the contribution from the constant forcing term ($f = \sigma$). For this forcing contribution, denoting its corresponding variation by $\delta_\sigma$, we have:
\begin{align}
\delta_\sigma g^{ij} &= -2\sigma h^{ij}, \\
\delta_\sigma h_{ij} &=  \sigma h_{il}h^l_j + \sigma \overline{R}_{0i0j} = \sigma h_{il}h^l_j - \sigma g_{ij},
\end{align}
where we used $\overline{R}_{0i0j} = -g_{ij}$ since the ambient space $M^3$ is a Fuchsian manifold equipped with a hyperbolic metric. We then evaluate the variation $\delta_\sigma |A|^2 = \delta_\sigma (g^{ik}g^{jl}h_{ij}h_{kl})$:
\begin{align}\label{eqn:cont-drift-var}
\delta_\sigma |A|^2 &= 2 (\delta_\sigma g^{ik}) h_i^l h_{kl} + 2 g^{ik}g^{jl} (\delta_\sigma h_{ij}) h_{kl} \notag\\
&= 2(-2\sigma h^{ik})(h^2)_{ik} + 2h^{ij}\big( \sigma (h^2)_{ij} - \sigma g_{ij} \big) \notag\\
&= -4\sigma \text{tr}(A^3) + 2\sigma \text{tr}(A^3) - 2\sigma H \notag\\
&= -2\sigma \text{tr}(A^3) - 2\sigma H.
\end{align}
Combining \eqref{eqn:cont-drift-var} with \eqref{eqn:evo-2nd-FF-MCF} yields \eqref{eqn:evol_A2_mod}.
\end{proof}

We next establish a polynomial growth estimate for the second fundamental form of the evolving surfaces, following the ideas from \cite{EH89}.

\begin{lemma} \label{lem:A2_time_dependent_bound}
Let $u \in C^{2,1}(\Sigma \times [0,T))$ be a smooth solution to the modified mean curvature flow \eqref{eqn:u_t} with $\sigma\in (-2,2)$. There exists a constant $C > 0$, depending only on the background metric $g_0$ on $\Sigma$, the $C^0$ bound of $u$ in Lemma \ref{lem:height_bound}, the constant $\sigma$, and the initial gradient bound $\|w(\cdot, 0)\|_\infty$, such that the second fundamental form of the evolving surface $\Sigma_t$ satisfies
\begin{equation}\notag
|A(x,t)| \le C(1+t)^{4}, \qquad \forall (x,t) \in \Sigma \times [0,\infty).
\end{equation}
\end{lemma}

\begin{proof}
Fix an arbitrary time $T > 0$. By Theorem \ref{thm:gradient_bound-sigma-0}, on $\Sigma \times [0, T]$ we have $w \leq C_1(1+T) := C_T$. We define the test function $Q = |A|^2 \Phi(w)$ with $\Phi(w) = w^2 e^{c_T w^2}$, where $c_T = \frac{1}{2C_T^2}$ such that $c_T w^2 \le \frac{1}{2}$ on $\Sigma \times [0, T]$.

We first compute:
\begin{align}
\Phi'(w) &= \left(\frac{2}{w} + 2c_T w\right) \Phi(w), \label{eqn:phi_prime-new} \\
\Phi''(w) &= \left(-\frac{2}{w^2} + 2c_T\right) \Phi(w) + \left(\frac{2}{w} + 2c_T w\right)^2 \Phi(w). \label{eqn:phi_double_prime-new}
\end{align}

Using \eqref{eqn:evol_w_modified}, the evolution equation of $\Phi(w)$ is given by:
\begin{align}\label{eqn:evol_phi}
(\partial_t - \Delta) \Phi &= \Phi' (\partial_t - \Delta) w - \Phi''|\nabla w|^2 \notag\\
&= -w\Phi'|A|^2 - \left(\frac{2}{w}\Phi' + \Phi''\right)|\nabla w|^2 + \Phi'\mathcal{E}_w,
\end{align}
where 
$$\mathcal{E}_w = 2w^2\tanh(u)H + 2\tanh(u)\langle \nabla w, \mathbf{n} \rangle  - \frac{w^2-1}{w\cosh^2(u)} - 2w\tanh^2(u) - \sigma\tanh(u)(w^2-1). $$ 
Using \eqref{eqn:evol_phi} and Lemma \ref{lem:evol_A}, the evolution equation of $Q$ is given by:
\begin{align}\label{eqn:evo-Q}
(\partial_t - \Delta) Q =& \Phi (\partial_t - \Delta)|A|^2 + |A|^2 (\partial_t - \Delta)\Phi - 2\nabla|A|^2 \cdot \nabla \Phi\\
=&\Phi \left(-2|\nabla A|^2 + 2|A|^4 + 4(|A|^2 - H^2) -2\sigma \text{tr}(A^3) - 2\sigma H\right) \notag\\
&+ |A|^2 \left(-w\Phi'|A|^2 - \left(\frac{2}{w}\Phi' + \Phi''\right)|\nabla w|^2 + \Phi'\mathcal{E}_w\right) - 2\nabla|A|^2 \cdot \nabla \Phi.\notag
\end{align}
Because $\Sigma \times [0, T]$ is compact, $Q$ achieves its maximum at some point $(x_0, t_0)$ and without loss of generality we can assume $t_0 \in (0, T]$. At $(x_0, t_0)$, we have $\nabla Q = 0$ (or equivalently $\nabla|A|^2 = -|A|^2 \frac{\Phi'}{\Phi}\nabla w$ or $\nabla|A| = -\frac{|A|}{2}\frac{\Phi'}{\Phi}\nabla w$) and $(\partial_t - \Delta) Q \ge 0$. Substituting $\nabla|A|^2= -|A|^2 \frac{\Phi'}{\Phi}\nabla w$ into the cross-term in \eqref{eqn:evo-Q} at $(x_0, t_0)$ yields:
\begin{equation}\notag
-2\nabla|A|^2 \cdot \nabla\Phi = -2\left(-|A|^2 \frac{\Phi'}{\Phi}\nabla w\right) \cdot (\Phi' \nabla w) = 2|A|^2 \frac{(\Phi')^2}{\Phi}|\nabla w|^2.
\end{equation}
Applying Kato's inequality $|\nabla|A|| \le |\nabla A|$, we have
\begin{equation}\notag
-2\Phi|\nabla A|^2 \le -2\Phi|\nabla|A||^2 = -2\Phi\left(\frac{|A|^2}{4}\frac{(\Phi')^2}{\Phi^2}|\nabla w|^2\right) = -\frac{1}{2}|A|^2 \frac{(\Phi')^2}{\Phi}|\nabla w|^2.
\end{equation}
We collect all gradient terms from \eqref{eqn:evo-Q} into a single expression $\mathcal{G}$:
\begin{align*}
\mathcal{G} &\le -\frac{1}{2}|A|^2 \frac{(\Phi')^2}{\Phi}|\nabla w|^2 - |A|^2 \left(\frac{2}{w}\Phi' + \Phi''\right)|\nabla w|^2 + 2|A|^2 \frac{(\Phi')^2}{\Phi}|\nabla w|^2 \notag\\
&= |A|^2 |\nabla w|^2 \left[ \frac{3}{2}\frac{(\Phi')^2}{\Phi} - \frac{2}{w}\Phi' - \Phi'' \right]\\
&= |A|^2 |\nabla w|^2 \Phi \left[ \frac{3}{2}y^2 - \frac{2}{w}y - \left(-\frac{2}{w^2} + 2c_T + y^2\right) \right]\\
&= -2c_T \Phi(1 - c_T w^2) |A|^2 |\nabla w|^2 \leq -c_T \Phi |A|^2 |\nabla w|^2\,,
\end{align*}
where $y = \frac{2}{w} + 2c_T w$ and we used \eqref{eqn:phi_prime-new}, \eqref{eqn:phi_double_prime-new} and $1 - c_T w^2 \geq \frac{1}{2}$ by the choice of $c_T$. Therefore, at the maximum $(x_0, t_0)$, using again \eqref{eqn:phi_prime-new} and \eqref{eqn:evo-Q}, we have
\begin{equation}\label{eqn:evol-Q-simplified}
0 \leq (\partial_t - \Delta) Q \le -2c_T w^2 \Phi |A|^4 - c_T\Phi |A|^2 |\nabla w|^2 + \Phi'\mathcal{E}_w|A|^2 + \Phi \mathcal{O}(|A|^3).
\end{equation}

Now we re-write
\begin{equation}\notag
\mathcal{E}_w = 2\tanh(u)\langle \nabla w, \mathbf{n} \rangle + \tilde{\mathcal{E}}_w,
\end{equation}
where $\tilde{\mathcal{E}}_w = 2w^2\tanh(u)H - \frac{w^2-1}{w\cosh^2(u)} - 2w\tanh^2(u) - \sigma\tanh(u)(w^2-1)$. Because $|u| \le C_0$ by Lemma \ref{lem:height_bound} and $|\mathbf{n}| \le 1$, we have
\begin{align}\label{eqn:young-ineq}
\left| 2\tanh(u) \Phi' |A|^2 \langle \nabla w, \mathbf{n} \rangle \right| \leq & C \Phi' |A|^2 |\nabla w| =  C \left(\frac{2}{w} + 2c_T w \right)\Phi |A|^2 |\nabla w| \notag\\
 \leq & 3C \Phi |A|^2 |\nabla w| = \left( \sqrt{2c_T \Phi} |A| |\nabla w| \right) \left( \frac{3C}{\sqrt{2c_T}} \sqrt{\Phi} |A| \right)\notag\\
\leq & c_T \Phi |A|^2 |\nabla w|^2 + \frac{9C^2}{4c_T} \Phi |A|^2.\notag
\end{align}
where we used \eqref{eqn:phi_prime-new}, $c_T w^2 \le \frac{1}{2}$ (so that $2c_T w = \frac{2c_T w^2}{w} \le \frac{1}{w} \leq 1$ since $w \ge 1$) and Young's inequality. The term $c_T \Phi |A|^2 |\nabla w|^2$ above cancels with the $-c_T\Phi |A|^2 |\nabla w|^2$ term in \eqref{eqn:evol-Q-simplified}. 

Now note that the leading term in $\tilde{\mathcal{E}}_w$ is $2w^2\tanh(u)H = \mathcal{O}(w^2|A|)$. Using \eqref{eqn:phi_prime-new} and multiplying $\tilde{\mathcal{E}}_w$ by $\Phi'|A|^2$ yields a bound by 
$$\Phi \left(\frac{2}{w} + 2c_T w\right) \mathcal{O}(w^2|A|^3) = \Phi \mathcal{O}((w + c_T w^3)|A|^3).$$ Since $w \le C_T$ and $c_T = \frac{1}{2C_T^2}$, we have $w + c_T w^3 \le C_T + \frac{1}{2}C_T = \frac{3}{2}C_T$. Thus, \eqref{eqn:evol-Q-simplified} at the maximum point $(x_0, t_0)$ reduces to:
\begin{equation}\notag
0 \le -2c_T w^2 \Phi |A|^4 + C_3 C_T\Phi |A|^3 + \frac{9C^2}{4c_T} \Phi |A|^2,
\end{equation}
where $C_3 > 0$ is a constant depending only on $|u|\leq C_0$ and $\sigma\in (-2,2)$. Therefore, using again $w\geq 1$, we have $|A(x_0, t_0)| \leq C_4 C_T^3$ (since $1/(2c_T) = C_T^2$).
Noting that $\Phi(w) \leq C_T^2 e^{1/2}$, we have
\begin{equation}
|A(x,t)|^2 \leq Q_{\max} = |A(x_0, t_0)|^2 \Phi(w(x_0, t_0)) \le \left(C_4 C_T^3\right)^2 (C_T^2 e^{1/2}) = C_5 (1+T)^{8}\,,
\end{equation}
 for all $(x,t) \in \Sigma \times [0, T]$. Since $T$ was chosen arbitrarily and the inequality holds for all $t \le T$, we conclude $|A(x,t)|^2 \le C_5(1+t)^{8}$, completing the proof.
\end{proof}
\begin{remark}\label{rmk:uniform-higher-order-bound}
    The proof of Lemma \ref{lem:A2_time_dependent_bound} also yields that if the gradient function satisfies a time-independent bound $w \le C_1$, then the norm of the second fundamental form satisfies a uniform bound $|A| \le C$ along the flow. It then follows directly from the standard induction argument of Huisken \cite[Lemma 7.2]{H86} (see also \cite{H84}) that all higher-order covariant derivatives of $A$ are uniformly bounded (noting that the constant normal forcing $\sigma \nu$ in the flow only generates extra terms of the form $\sigma \sum_{i+j=m} \nabla^i A * \nabla^j A * \nabla^m A$ and $\sigma \nabla^m \overline{R}$ in the evolution equation for $|\nabla^m A|^2$).
\end{remark}

\subsubsection{Uniform gradient bound via blow-up}
We establish a uniform gradient bound by using the exponential $C^0$ convergence of the flow to the equidistant surface $\Sigma(r_{\sigma})$ by Corollary \ref{cor:exponential_bound}, combined with a Type II curvature blow-up argument.

\begin{theorem}\label{thm:gradient_bound_general}
Let $u \in C^{2,1}(\Sigma \times [0,\infty))$ be a smooth solution to the modified mean curvature flow \eqref{eqn:u_t} with $\sigma\in (-2,2)$. There exists a constant $C > 0$, depending only on the background metric $g_0$ on $\Sigma$, the $C^0$ bound of $u$ in Lemma \ref{lem:height_bound}, the constant $\sigma$, and the initial gradient bound $\|w(\cdot, 0)\|_\infty$, such that:
\begin{equation}\notag
w(x,t) \le C\,, \qquad \forall (x,t) \in \Sigma \times [0,\infty).
\end{equation}
\end{theorem}

\begin{proof}
We proceed by contradiction. Assume the uniform gradient bound fails. Then by Theorem \ref{thm:gradient_bound-sigma-0}, we have 
\begin{equation}\label{eqn:blow-up-of-gradient}
\limsup_{t \to \infty} \max_{x \in \Sigma} w(x,t) = \infty.
\end{equation}
We claim that we must also have
$$\limsup_{t \to \infty} \max_{x \in \Sigma} |A|(x,t) = \infty.$$

Suppose for contradiction that $|A(x,t)| \le K$ uniformly for all $x\in \Sigma$ and $t > 0$. Then following Huisken \cite[Lemma 7.2]{H86} (see also Remark \ref{rmk:uniform-higher-order-bound}), all higher-order covariant derivatives of $A$ are also uniformly bounded: $|\nabla^m A| \le C_m$. Since the evolving surfaces $\Sigma_t$ are graphs over a compact surface $\Sigma$, uniform bounds on $A$ and all its covariant derivatives, combined with the uniform height bound (Lemma \ref{lem:height_bound}), imply that the height function $u(\cdot, t)$ is uniformly bounded in $C^k(\Sigma)$ for all integers $k \ge 0$. Moreover, by Lemma \ref{lem:height_bound}, the flow converges exponentially in $C^0(\Sigma)$ to the equidistant surface $u \equiv r_\sigma = \text{arctanh}(\sigma/2)$. Standard Gagliardo-Nirenberg interpolation inequalities guarantee that a sequence uniformly bounded in $C^2(\Sigma)$ and converging in $C^0(\Sigma)$ must also converge in $C^1(\Sigma)$. Consequently, the spatial gradient $Du(\cdot, t)$ converges uniformly to $0$ and $w = \sqrt{1 + |Du|_{g_0}^2/\cosh^2(u)}$ converges uniformly to $1$ everywhere on $\Sigma$ as $t \to \infty$, which contradicts the assumption \eqref{eqn:blow-up-of-gradient}.

Therefore, there exists a sequence of times $t_k \to \infty$ and points $x_k \in \Sigma$ such that:
\begin{equation}
A_k := |A(x_k, t_k)| = \max_{x \in \Sigma, \, t \le t_k} |A(x,t)| \to \infty.
\end{equation}
Note that the points $(x_k,t_k)$ achieving this maximum of $|A|$ need not coincide with the points maximizing the gradient $w$.

Now we translate the space-time origin to $p_k = (x_k, u(x_k, t_k), t_k)$ and perform a parabolic rescaling by the factor $\lambda_k = A_k$. Let $u_k = u(x_k, t_k)$. We define the rescaled height function $\tilde{u}_k: \tilde{\Omega}_k \times [\tilde{t}_{start}, 0] \to \mathbb{R}$ by:
\begin{equation}
\tilde{u}_k(\tilde{x}, \tilde{t}) = A_k \left( u(x_k + A_k^{-1}\tilde{x}, t_k + A_k^{-2}\tilde{t}) - u_k \right),
\end{equation}
where $\tilde{\Omega}_k = A_k (\Sigma - x_k)$ and $\tilde{t}_{start} = -A_k^2 t_k$. The rescaled surface is the graph $\tilde{\Sigma}_k^{\tilde{t}} = \{ (\tilde{x}, \tilde{u}_k(\tilde{x}, \tilde{t})) : \tilde{x} \in \tilde{\Omega}_k \}$. 

By definition, the rescaled surfaces satisfy $|\tilde{A}_k(\tilde{x}, \tilde{t})|^2 = A_k^{-2}|A|^2 \le 1$ for all $\tilde{t} \in [\tilde{t}_{start}, 0]$, and attain exactly $|\tilde{A}_k(0,0)| = 1$. We then apply a global rotation $R_k \in SO(3)$ to the ambient space such that $R_k(\nu_k) = \mathbf{n}$, where $\nu_k$ is the unit normal of the evolving surface at $p_k$. 
Let $\hat{\Sigma}_k^{\tilde{t}} = R_k(\tilde{\Sigma}_k^{\tilde{t}})$. As $A_k \to \infty$, the ambient space flattens to Euclidean space $\mathbb{R}^3$. Because $|\tilde{A}_k(\tilde{x}, \tilde{t})|^2 \le 1$ for all $\tilde{t} \in [\tilde{t}_{start}, 0]$, all higher-order derivatives of $\hat{A}_k$ are uniformly bounded (see Remark \ref{rmk:uniform-higher-order-bound}) which guarantees that a subsequence converges smoothly to an ancient limit flow $\hat{\Sigma}_\infty^{\tilde{t}}$ in $\mathbb{R}^3$, defined for $\tilde{t} \in (-\infty, 0]$, with
\begin{equation}\label{eqn:limit_curvature_one}\notag
|\hat{A}_\infty(0,0)| = 1.
\end{equation}

Now by Lemma \ref{lem:A2_time_dependent_bound}, we have
\begin{equation}\label{eqn:polynomial_curvature_bound}\notag
A_k = \max_{x \in \Sigma, \, t \le t_k} |A(x,t)| \le C (1 + t_k)^4.
\end{equation}
Note that by Lemma \ref{lem:height_bound}, the original evolving surface $\Sigma_{t_k}$ lies within a horizontal slab of width less than $C e^{-\frac{4-\sigma^2}{2} t_k}$. Under the parabolic rescaling by $\lambda_k = A_k$, the rescaled width of this slab is bounded by
\begin{equation}\notag
A_k \cdot C e^{- \frac{4-\sigma^2}{2}  t_k}  \leq C (1+t_k)^4 e^{- \frac{4-\sigma^2}{2}  t_k} \to 0 \quad \text{as }t_k \to \infty.
\end{equation}

The rotations $R_k$ preserve ambient distances and thus, the smooth ancient limit surface $\hat{\Sigma}_\infty$ must be a flat plane, which contradicts $|\hat{A}_\infty(0,0)| = 1$. Therefore, the assumption \eqref{eqn:blow-up-of-gradient} must be false. The gradient $w(x,t)$ is uniformly globally bounded for all time.
\end{proof}

\begin{proof}[Proof of Theorem \ref{main-thm}]
By Theorem \ref{thm:gradient_bound_general}, the gradient function $w$ of the evolving surfaces $\Sigma_t$ is globally and uniformly bounded as long as the flow exists. Combined with the uniform height bound from Lemma \ref{lem:height_bound}, this guarantees that the second fundamental form and all its higher-order covariant derivatives are uniformly bounded globally in space and time (c.f. Remark \ref{rmk:uniform-higher-order-bound}). This precludes the formation of finite-time singularities, ensuring that the modified mean curvature flow exists smoothly for all time $t \in [0, \infty)$ and that the evolving surfaces remain geodesic graphs over $\Sigma$.

By Corollary \ref{cor:exponential_bound}, the height function of the flow converges exponentially in $C^0(\Sigma)$ to the equidistant surface $\Sigma(r_\sigma)$, where $r_\sigma = \text{arctanh}\left(\frac{\sigma}{2}\right)$. Because the height function and all its higher-order derivatives are uniformly bounded for all time, standard interpolation inequalities guarantee that this exponential $C^0$ convergence elevates to smooth convergence in the $C^\infty$ topology. Therefore, the flow converges smoothly to $\Sigma(r_\sigma)$ as $t \to \infty$.
\end{proof}

\bibliographystyle{amsplain}
\bibliography{ref}
\end{document}